\documentclass[11pt]{amsproc}
\usepackage{amsmath, amstext, amsbsy, amssymb}
\usepackage{graphicx}
\usepackage{subfigure}
\usepackage{picinpar}
\usepackage{psfrag}
\usepackage{color,xcolor}
\input xy \xyoption{all}

\numberwithin{equation}{section}
\def\beq{\begin{eqnarray}}
\def\eeq{\end{eqnarray}}
\def\beqs{\begin{eqnarray*}}
\def\eeqs{\end{eqnarray*}}

\def\mz{{\mathbb Z}}
\def\mc{{\mathbb C}}
\def\mr{{\mathbb R}}

\def\ca{{\mathcal A}}

\def\gr{{\rm Gr}}

\def\bem{\begin{array}{lll}}
\def\eem{\end{array}}

\def\dim{{\hbox{\rm dim}}}

\def\ad{{{\rm ad}}}

\newfont{\df}{eufm10}

\def\mc{{\mathbb C}}

\def\sa{\frak{a}}
\def\sb{\frak{b}}
\def\sg{\frak{g}}

\def\sp{\frak{p}}

\def\sn{\frak{n}}

\def\sh{\frak{h}}
\def\sss{\frak{s}}
\def\sf{\frak{f}}
\def\ssf{{\bf f}}
\def\ssx{{\bf x}}
\def\bh{{H}}
\def\sk{{\frak{k}}}

\def\sbo{\overline{\sb}}
\def\sno{\overline{\sn}}

\newtheorem{theo}{Theorem}[section]

\newtheorem{defi}[theo]{Definition}
\newtheorem{lemma}[theo]{Lemma}

\newtheorem{prop}[theo]{Proposition}
\newtheorem{remark}[theo]{Remark}

\title[]
{Whittaker modules and hyperbolic Toda lattices}

\author[L. Xia]{Limeng Xia}
\address{XIA: College of Mathematical Sciences, Jiangsu University, XueFu Road 301,
Zhenjiang 212013, Jiangsu,  PR China} \email{xialimeng@ujs.edu.cn}

\date{}

\begin{document}
\maketitle

\begin{abstract} Let $\sg$ be a complex finite-dimensional simple Lie algebra and let $\sg_l$ be the corresponding generalized Takiff algebra. This paper studies the affine variety $\ssf+\sb_l$ where $\ssf$ is similar to a principal nilpotent element of $\sg$ and $\sb_l$ is a subalgebra corresponding to the Borel subalgebra $\sb$ of $\sg$. Inspired by Kostant's work then we deal with two questions. One of them is to construct the Whittaker model for the $G_l$-invariants of symmetric algebra $S(\sg_l)$ where $G_l$ is the adjoint group of $\sg_l$ and $G_l$ acts on $S(\sg_l)$ by coadjoint action, and then to classify all nonsingular Whittaker modules over $\sg_l$. Another one is to describe the symplectic structure of the manifold $Z\subseteq\ssf+\sb_l$ of normalized Jacobi elements. Then the Hamiltonian corresponding to  a fundamental invariant provides a class of hyperbolic Toda lattices. In particular, a simplest example describes the state of a dynamical system consisting of a positive mass particle and a negative mass particle.
\end{abstract}



\smallskip\bigskip

\section{Introduction}
This paper deals with a representation question as to which irreducible representations of a Takiff algebra have a Whittaker model, and with a question involving some Toda lattices which describes  the dynamical systems consisting of  both positive mass articles and negative mass particles. These Toda lattices attached to Takiff algebras were also studied by M. Lau in the very recent publication \cite{Lau}.

\subsection{} In his original paper Kostant suggested the so-called {\it Whittaker model} of the center $Z(\sg)$ of the universal enveloping algebra $U(\sg)$ of a finite-dimensional complex semi-simple Lie algebra $\sg$ in \cite{Kos3} in 1978. He studied the variety $\sf+\sb$ where $\sf$ is a principal nilpotent element and $\sb$ is the Borel subalgebra. This variety induces a translated adjoint action of $N$ on $\sb$ and hence a translated coadjoint action on $\sb_l$, where $N$ is the Lie subgroup corresponding to the nilradical $\sn$ of $\sb$. This translated coadjoint action determines an isomorphism between the $G$-invariants $S(\sg)^G$ and the $N$-invariants $S(\sbo)^N$. Then one constructs the Whittaker model. As a main result in the representation theory then Kostant classified all nonsingular Whittaker modules over $\sg$.

A more general definition of a Whittaker module for an algebra structure $\mathcal{A}$ with a triangular decomposition $\mathcal{A}=\mathcal{A}^-\mathcal{A}^0\mathcal{A}^+$ is then an  $\mathcal{A}$-module induced from a one-dimensional $\mathcal{A}^+$-module which is determined by some character $\mathcal{A}^+\rightarrow\mc$.

In the recent decades many authors have focused on the Whittaker models and Whittaker modules over various algebras including Lie algebras, Lie super algebras, vertex operator algebras  and quantum groups.  For example, Adamovi\'{c}-L\"{u}-Zhao studied  Whittaker modules for the affine Lie algebra of type $A^{(1)}$ in \cite{ALZ}. Christodoulopoulou studied the so-called imaginary Whittaker modules for affine Lie algebras induced by the Whittaker modules for infinite-dimensional Heisenberg algebras in \cite{C}. Hartwig-Yu studied  Whittaker modules over free bosonic orbifold vertex operator algebras in \cite{HY}. Liu-Pei-Xia studied Whittaker modules for the super-Virasoro algebras in \cite{LPX}. Sevostyanov studied  Whittaker modules over topological quantum groups in \cite{S} and Xia-Zhao studied twisted Whittaker modules over rational quantum groups of type $A$ in \cite{XZ}. Romanov developed a geometric algorithm for computing the composition multiplicities of stantard Whittaker modules in \cite{Rom}.

\subsection{}

In 1979 by considering the coadjoint action of $\overline{B}$ on the variety $\sf+\sb$ Kostant studied the symplectic structure of the manifold $Z\subseteq\sf+\sb$ of normalized Jacobi elements in \cite{Kos5}(also see \cite{Kos4}), where $\overline{B}$ is the Lie subgroup corresponding to the opposite Borel subalgebra $\sbo$. Let $\kappa$ be a positive multiple of the Killing form of $\sg$. The invariant $I=\frac12\kappa$ then produces a generalized Toda lattice which is related to the Dynkin diagram of $\sg$. Of course the Toda lattices relative to Lie algebras have attracted many attentions. For example, one is referred to  \cite{LNW} and \cite{N}.

In this paper, the Toda lattice means a finite non-periodic Toda lattice. A Toda lattice is a Hamiltonian system describing particles on a line which was first introduced by Toda in \cite{Toda}. The Hamiltonian of  a system of $n$ particles moving on a line so that the particles have masses ${\bf m}_i>0, i=1,\cdots,n$ usually is written as
\beqs H=\sum_{i=1}^n\frac{p_i^2}{2{\bf m}_i}+U,\eeqs
where $U$ is the potential function, then the state of this system can be determined by a Hamilton's equation. Usually this problem was studied by a symplectic structure and the Hamiltonian defines a bilinear form  on certain real space.

In theoretical physics, the possibility of having negative-mass was already present from Dirac established his famous equation in \cite{D}. Of course many physicists have been concerned with problems involving negative-mass matters, for example see \cite{PdA}.

Considering a simplest Hamiltonian system which consists of a positive mass particle and a negative mass particle, then one can assume
\beq\label{h-example-1}\bh=\frac{(p_+)^2}{2}-\frac{(p_-)^2}{2}+U(q_+,q_-)\eeq
so that
\beq\label{h-example-2}\frac{\partial\bh}{\partial p_\pm}=\frac{dq_{\pm}}{dt},\;\frac{\partial\bh}{\partial q_\pm}=-\frac{dp_{\pm}}{dt}.\eeq
If we assume $p=(p_++p_-)/2, \overline{p}=(p_+-p_-)/2, q=q_++q_-, \overline{q}=q_+-q_-$, then one has
\beq\label{h-example-3} &&\bh=2p\overline{p}+U,\\
\label{h-example-4}&&\frac{\partial\bh}{\partial p}=\frac{d{\overline{q}}}{dt},\;\frac{\partial\bh}{\partial \overline{p}}=\frac{d{q}}{dt},\;\frac{\partial\bh}{\partial q}=-\frac{d{p}}{dt},\;\frac{\partial\bh}{\partial \overline{q}}=-\frac{d\overline{p}}{dt}.\eeq

Clearly the corresponding symplectic structure induces a non-degenerate symmetric bilinear form on $\mr q_++\mr q_-$ determined by
\beqs (q_\pm,q_\pm)=\pm1, (q_\pm,q_\mp)=0.\eeqs
Then one has
\beqs (q,q)=(\overline{q},\overline{q})=0, (q,\overline{q})=2.\eeqs

In general, if a system contains both positive mass particles and negative mass particles, it is impossible to determine the corresponding bilinear form on the position-phase space $\sum_{i=1}^n\mr q_i$ by a multiple (even negative multiple) of the Killing form of some   semi-simple Lie algebra.

\subsection{} Suppose that $\sg$ is a finite-dimensional complex simple Lie algebra of rank $n$ and $\sg_l=\sg\otimes\mc[t]/(t^{l+1})$ is the Takiff algebra. The Takiff algebra is an important non-semi-simple object of finite-dimensional Lie algebras. Of course lots of articles provides many results in its representation theory. For more details one is referred to \cite{CKR}, \cite{CC}, \cite{MM} and \cite{W} and references therein.

Let $Z(\sg_l)$ be the center of the universal enveloping algebra $U(\sg_l)$ of $\sg_l$. For any $U(\sg_l)$-module $V$ let $U(\sg_l)_V$ denote the annihilator of $V$ in $U(\sg_l)$ and put $Z(\sg_l)_V=Z(\sg_l)\cap U(\sg_l)_V$. One says that $V$ admits an infinitesimal character if $Z(\sg_l)_V$ is a maximal ideal of $Z(\sg_l)$. Clearly $V$ admits an infinitesimal character if and only if there exists a character $\xi:Z(\sg_l)\rightarrow\mc$ such that $Z(\sg_l)_V={\rm Ker}\xi$.

For any subspace $\frak{a}\subseteq\sg$ let $\frak{a}_l$ denote the subspace $\frak{a}\otimes\mc[t]/(t^{l+1})\subseteq\sg_l$. Let $\sh$ be a Cartan subalgebra of $\sg$ and let $\sb, \sbo$ respectively be the positive Borel subalgebra and the opposite Borel subalgebra. Then both $\sn=[\sb,\sb]$ and $\overline{\sn}=[\sbo,\sbo]$ are maximal nilpotent subalgebras. In particular one has a triangular decomposition
\beqs\sg_l=\overline{\sn}_l\oplus\sh_l\oplus\sn_l.\eeqs

If $0\not=v\in V$ and $\eta:\sn_l\rightarrow\mc$ is a character such that
\beqs xv=\eta(x)v,\;\forall x\in\sn_l,\eeqs
the element $v$ is called a Whittaker vector (of type $\eta$). If the module $V$ is cyclically generated by a Whittaker vector, it is called a Whittaker module.

Let $\mc_{\eta}$ be the one-dimensional $U(\sn_l)$-module defined by the character $\eta$ of $\sn_l$ and  let $\mc_{\eta,\xi}$ be the one-dimensional $Z(\sg_l)\otimes U(\sn_l)$-module defined by the respective characters $\xi$ and $\eta$ of $Z(\sg_l)$ and $\sn_l$. Considering the induced modules
\beqs &&W_\eta=U(\sg_l)\otimes_{U(\sn_l)}\mc_\eta,\\
&&Y_{\eta,\xi}=U(\sg_l)\otimes_{Z(\sg_l)\otimes U(\sn_l)}\mc_{\eta,\xi}.\eeqs

We prove that there exists a one-one correspondence between the set of all submodules of $W_\eta$ and the set of all ideals of the center $Z(\sg_l)$. In particular, the modules $Y_{\eta,\xi}$ ($\eta$ nonsingular) classify all simple Whittaker modules up to equivalence. These results generalize   Kostant's corresponding results from $\sg_l$ to Takiff algebra $\sg_l$.

\subsection{}

It is known that $\sg_l$ has a non-degenerate invariant symmetric bilinear form $Q$ induced by the Killing form of $\sg$. Due to this bilinear form, the symmetric algebra $S(\sg_l)$ of $\sg_l$ can be regarded as the algebra of polynomial functions on $\sg_l$ and then $S(\sg_l)$ has a Poisson structure. Moreover the restriction of this structure to $S(\sbo_l)$ coincides with the Poisson structure of $S(\sbo_l)$ as the algebra of polynomial functions on $\sb_l$. Let $G_l$ be the adjoint group of $\sg_l$ and let $\overline{B}_l$ be the Lie subgroup  corresponding to $\sbo_l$ then $\overline{B}_l$ acts on $\sb_l$ by the coadjoint action. 
Considering the manifold $Z\subseteq \ssf+\sb_l$ of normalized Jacobi elements, then there exists a modified version of the coadjoint action of $\overline{B}_l$ on $Z$ so that $Z$ is a $\overline{B}_l$-orbit in $\ssf+\sb_l$, and hence  $Z$ has a coadjoint-symplectic theory of $\overline{B}_l$ carried from $\sb_l$.

In particular one determines the symplectic structure of $Z$ by its a global ordinate system. Let $I=\frac12Q$, one obtains a Hamiltonian corresponding to $I$, which is of hyperbolic type. In the case $l=1$ we give a simple expression for the $2$-form $\omega_Z$ of $Z$. The Hamiltonian corresponding to $I$ then is clearly of hyperbolic type. The associative systems are completely integrable and referred to the hyperbolic Toda lattices.

Finally, we give an example in the case $n=1,l=1$. In this case one obtains a hyperbolic differential operator on $\mr\times\mr^+$ and an explicit example for \eqref{h-example-3}-\eqref{h-example-4} from the Hamiltonian corresponding to $I=\frac12Q$. In particular, one shows that the dynamical system determined by the Hamiltonian in this example has a unique global smooth solution for any initial values under some bounded condition.

\subsection{} Throughout this paper, we  always let $\mc, \mc^*, \mr, \mr^*, \mz, \mathbb{N}$ denote the set of complex numbers, the set of nonzero complex numbers, the set of real numbers, the set of nonzero real numbers, the set of integers and the set of non-negative integers, respectively.

In Sections 2-5 and Subsection 6.3 Lie algebras $\sg$ and $\sg_l$ are defined over ground field $\mc$. In Sections 6-7 (except Subsection 6.3) the same Lie algebras are regarded as Lie algebras over ground field $\mr$.

\section{Preliminaries}

\subsection{}For any finite-dimensional simple complex Lie algebra $\sg$  let $G$ denote the adjoint group of $\sg$. Let $ax\in\sg_l$ denote the adjoint action of $a$ on $x$ for any $a\in G$ and $x\in\sg$.

Fixing  $\kappa$ as a positive multiple of the Killing form of $\sg$,  then via $\kappa$ one can identify $\sg$ as its dual space $\sg^*=\hom_\mc(\sg,\mc)$ by $x\mapsto \kappa_x$ where $\kappa_x(y)=(x,y)$ for all $y\in\sg$. It is convenient to  write $x(y)=\kappa_x(y)$ without confusion. The action of $G$ on $\sg$ extends to $S(\sg)$ as a group of automorphisms. Let $S(\sg)$ be the symmetric algebra of $\sg$ and let $S_k(\sg)$ denote the homogeneous space of degree $k$.

Let $S(\sg)^G$ be the algebra of $G$-invariants in $S(\sg)$. By Chevally's theorem there are homogeneous elements $I_j\in S_{m_j+1}(\sg)$  such that
\beq\label{center-g} S(\sg)^G=\mc[I_1,\cdots,I_n]\eeq
is a polynomial algebra in the $I_j$. Where the $m_j$ are called the exponents of $G$ (one is referred to \S 4.4 in \cite{Kos2}).

Let $\sh\subseteq\sg$ be a Cartan subalgebra and let $\Pi=\{\alpha_1,\cdots,\alpha_n\}$ be a prime root system. Then the root system $\Delta$ and the positive root system $\Delta_+$ are uniquely determined. The Cartan subalgebra $\sh$ has a basis $\{\omega_1,\cdots,\omega_n\}$ such that $\alpha_j(\omega_i)=\delta_{i,j}$ and another basis $\{h_1,\cdots,h_n\}$ such that $\alpha_j(h_i)=c_{i,j}$ where
\beq\label{cartan-matrix} C=(c_{i,j})_{n\times n}\eeq
is the Cartan matrix of $\sg$.

Associate a root vector $e_\alpha$ to each root $\alpha\in\Delta$. Then
\beqs \sb=\sh+\sum_{\alpha\in\Delta_+}\mc e_\alpha,\; \sbo=\sh+\sum_{\alpha\in\Delta_+}\mc e_{-\alpha}\eeqs
are the associated Borel subalgebra and the opposite Borel subalgebra, respectively. They have their respective nilradicals
\beqs \sn=[\sb,\sb],\;\sno=[\sbo,\sbo].\eeqs

Let $J$ be the number of positive roots. We arrange the positive roots as $\{\beta_i|1\leq i\leq J\}$ such that $\beta_i=\alpha_i$ for any $1\leq i\leq n$.


\subsection{} We define the elements
\beqs \sf=\sum_{i=1}^n e_{-\alpha_i},\;x_0=\sum_{i=1}^n\omega_i.
\eeqs

Clearly $x_0$ is the unique element in $\sh$ such that $\alpha_i(x_0)=1$ for all $1\leq i\leq n$. Let $\sg_{[i]}$ be the eigenspace for $\ad x_0$ corresponding to the eigenvalue $i$, then
\beq\label{grading-g}\sg=\bigoplus_i\sg_{[i]}\eeq
and $[\sg_{[i]},\sg_{[j]}]\subseteq \sg_{[i+j]}$. In particular $\sf\in\sg_{[-1]}$ and $[\sf,\sn]\subseteq\sb$.

Noting that both $\sb$ and $[\sf,\sn]$ are stable under the action of $\ad x_0$. Since $\ad x_0$ is semi-simple there exists an $\ad x_0$-stable subspace $\sss\subseteq\sb$ such that
\beq\label{b=fn+s}\sb=[\sf,\sn]\oplus\sss\eeq
is a direct sum (see (1.1.5) in \cite{Kos3}).

Let $N\subseteq G$ be the (unipotent) Lie subgroup corresponding to $\sn$. Regarding $N, \sf+\sss$ and $\sf+\sb$ as affine varieties then one has
\begin{theo}[Theorem 1.2 in \cite{Kos3}]\label{variety-iso-1} The map
\beqs N\times(\sf+\sss)\rightarrow\sb\eeqs
given by $(a,x)\mapsto ax$ is an isomorphism of affine varieties.
\end{theo}

\begin{remark}The element $\sf$ is a principal nilpotent element in the notation of \cite{Kos2}. In fact, an element $\sf=\sum_{i=1}^nc_ie_{-\alpha_i}$ is principal nilpotent if and only if $c_i\not=0$ for all $c_i$. In this case \eqref{b=fn+s} always holds for some $\ad x_0$-stable subspace $\sss\subseteq\sb$ and Theorem \ref{variety-iso-1} is also true.
\end{remark}


\subsection{} For any positive integer $l$ the truncated polynomial current Lie algebra $\sg_l$ is defined to be the tensor product $\sg\otimes\ca_l$, where $\ca_l=\mc[t]/\mc[t]t^{l+1}$. Such algebras were studied in the case $l=1$ by S. J. Takiff in \cite{T} and then the Lie algebra $\sg_l$ is called the generalized Takiff algebra (see \cite{M}).

For convenience, for any subspace $S\subseteq\sg$ and any element $x\in\sg$, we denote by
\beqs S_l=S\otimes\ca,\;S(k)=S\otimes t^k,\;S[k]=S\otimes t^k\ca_l,\;x(k)=x\otimes t^k.\eeqs

There exists a unique non-degenerate invariant bilinear form $Q$ on $\sg_l$ defined by
\beqs Q(x(i),y(j))=\delta_{i+j,l}\kappa(x,y),\quad \forall x,y\in\sg, 0\leq i,j\leq l.\eeqs
We identify $\sg_l$ as its dual space via $x\mapsto Q_x=Q(x,-)$. Let $S(\sg_l)$ be the symmetric algebra of $\sg_l$ and $S_k(\sg_l)$ denotes the homogeneous space of degree $k$.

Let $G_l$ denote the adjoint group of $\sg_l$. For any $a\in G_l$ and $x\in\sg_l$, $ax\in\sg_l$ denotes the adjoint action of $a$ on $x$. The action of $G_l$ on $\sg_l$ extends to $S(\sg_l)$ as a group of automorphisms. Let $S(\sg_l)^{G_l}$ be the subalgebra of $G_l$-invariants in $S(\sg_l)$.

For any $x\in\sg$ let $\overline{x}=\sum_{i=0}^lx(i)z^i\in S(\sg_l)[z]$ in a variable $z$. We may write $I_i=\sum_{x}x_1\cdots x_{m_i+1}$, let
\beq\label{invariants-S(g-l)} I_i(z)=\sum_{j=0}^{il}I_{i,j}z^{il-j}=\sum_{x}\overline{x_1}\cdots \overline{x_{m_i+1}}.\eeq
Then $I_{i,j}\in S_{m_i+1}(\sg_l)^{G_l}$ for $0\leq j\leq l$. In particular, the subalgebra $S(\sg_l)^{G_l}$ is a polynomial algebra in the homogeneous $G_l$-invariants
\beq\label{center-gl}I_{i,j},\;1\leq i\leq n, 0\leq j\leq l.\eeq
For more details on the $G_l$-invariants one is referred to \cite{M}, \cite{RT} and also \cite{T}.

\section{A Decomposition result of $S(\sbo_l)$}

\subsection{} For any $0\leq j\leq l$ fixing a principal nilpotent element $\sf_j\in\sg_{[-1]}$. We define elements
\beq\label{fl-xl} \ssf=\sum_{j=0}^l\sf_j(j),\;\ssx=\sum_{i=0}^n\omega_i(0).\eeq
Let $\sss\subseteq \sb$ be an $\ad x_0$-stable subspace such that $\sb=[\sf_0,\sn]\oplus\sss$ is a direct sum.

For each $j$, obviously the subspace ${\sg_{[j]}}_l$ is the eigenspace for $\ad\ssx$ corresponding to the eigenvalue $j$ and one has
\beq\sg_l=\bigoplus_j{\sg_{[j]}}_l.\eeq
In particular one has $\ssf\in{\sg_{[-1]}}_l$, $\sss_l$ is $\ad\ssx$-stable.
\begin{lemma}\label{direct-sum}One has
\beq\sb_l=[\ssf,\sn_l]+\sss_l\eeq
and the sum is direct.
\end{lemma}
\begin{proof}Clearly one has
\beqs &&\sb[l]=\sb(l)=[\sf_0,\sn](l)+\sss(l)=[\ssf,\sn(l)]+\sss(l),\\
&&\sb[l-1]=[\sf_0,\sn](l-1)+\sss(l-1)+\sb(l)=[\ssf,\sn(l-1)]+\sss(l-1)+\sb[l],\eeqs
and \beqs
&&\sb[k]=\sb(k)+\sb[k+1]=[\sf_0,\sn](k)+\sss(k)+\sb[k+1]=[\ssf,\sn(k)]+\sss(k)+\sb[k+1],
\eeqs
for all $k\geq0$. Then this lemma follows from the direct sum
\beqs \sb_l=\sum_{j=0}^l\sb(j)\eeqs
since
\beqs \sb_l=\sum_{j=0}^l\sb(j)=\sum_{j=0}^l\left([\ssf,\sn(j)]+\sss(j)\right)=[\ssf,\sn_l]+\sss_l.\eeqs
\end{proof}

Let $N_l\subseteq G_l$ be the (unipotent) Lie subgroup corresponding to $\sn_l$. Regarding $N_l, \ssf+\sss_l$ and $\ssf+\sb_l$ as affine varieties then one has
\begin{theo}\label{variety-iso-2} The map
\beq\label{variety-iso-eq-2} N_l\times(\ssf+\sss_l)\rightarrow\ssf+\sb_l\eeq
given by $(a,x)\mapsto ax$ is an isomorphism of affine varieties.
\end{theo}
\begin{proof}
First we show that $a(\ssf+x)=\ssf+x$ for $x\in\sb_l$ and $a\in\sn_l$ implies that $a$ is the identity.

We may write $a=\exp(z)$ for some $z=\sum_{j\geq0}z_j(j)$ and  $x=\sum_{j}x_j(j)$ where $z_j\in\sn, x_j\in\sb$. Then one has
\beqs \ssf+x=a(\ssf+x)\in \exp(z_0(0))(\sf(0)+x_0(0))+\sum_{j>0}\sg(j)\eeqs
and hence $z_0=0$ by Theorem \ref{variety-iso-1}. Assume that $z=\sum_{j\geq 1}z_j(j)$. Then  one has
\beqs \ssf+x=a(\ssf+x)\in f_0(0)+x_0(0)+f_1(1)+x_1(1)+[z_1(1),\sf_0(0)+x_0(0)]+\sum_{j>1}\sg(j)\eeqs
and hence $z_1=0$. By induction in the similar way one has $z=0$ and $a$ is the identity.

Next we show that $a(\ssf+x)=(\ssf+x')$ for some $x, x'\in\sss_l$ and $a\in\sn_l$ implies that $x=x'$ and  $a$ is the identity. We write $a=\exp(z)$ for some $z=\sum_{j\geq0}z_j(j), x=\sum_{j\geq0}x_j(j), x'=\sum_{j\geq0}x_j'(j)$ where $z_j\in\sn, x_j, x_j'\in\sss$. Then
\beqs \ssf+x'=a(\ssf+x)\in \exp(z_0(0))(\sf_0(0)+x_0(0))+\sum_{j>0}\sg(j).\eeqs
Let $N(0)\subseteq G_l$ be the Lie subgroup corresponding to the subalgebra $\sn(0)$. Then by Theorem \ref{variety-iso-1}, $\ssf+x_0'(0)\in N(0)(\ssf+x_0(0))$ if and only if $x_0=x_0'$. In particular it also implies $z_0=0$. Now let $i$ be minimal such that $z_{i}\not=0$ then one has
\beqs \ssf+x'=a(\ssf+x)\in \ssf+x+[z_i(i),\sf_0(0)+x_0(0)]+\sum_{j>i}\sg(j)\eeqs
and hence $z_i=0$. So $z=0$ and $a$ is the identity. Of course one has $x=x'$.

Up to now we have proved that the map is injective. Finally we show that the map is surjective.  For each $y\in\sb_l$, \ there exists some element $x\in\sss_l$ such that $\ssf+y\in N_l(\ssf+x)$. By Lemma \ref{direct-sum}, we may assume
\beqs y=\sum_{i\geq0}y_i+\sum_{i\geq0}s_i,\eeqs
where $y_i\in[\ssf,\sn_l]\cap{\sg_{[i]}}_l, s_i\in\sss_l\cap{\sg_{[i]}}_l$ are uniquely determined. Noting that $\dim[\ssf,\sn_l]=\dim\sb_l-\dim\sss_l=\dim\sn_l$. Let $i_0$ be the minimal such that $y_{i_0}\not=0$.  Then there exists a unique element $z\in\sn_l\cap{\sg_{[i_0+1]}}_l$ such that $y_{i_0}=[\ssf,z]$ and hence
\beqs (\ad z)(\ssf+y)\in\ssf+y-y_{i_0}+\sum_{i>i_0}{\sg_{[i]}}_l.\eeqs
Let $y'=(\ad z)(\ssf+y)$, then one has
\beqs y'=\sum_{i\geq i_0+1}y_i'+\sum_{i\geq0}s'_i\eeqs
where $y_i'\in[\ssf,\sn_l]\cap{\sg_{[i]}}_l, s_i'\in\sss_l\cap{\sg_{[i]}}_l$. By induction on $i_0$, one can obtain some element $y''\in\sss_l$ such that $y$ belongs to the $N_l$-orbit of $y''$ and hence the map \eqref{variety-iso-2} is surjective.
\end{proof}

\begin{remark}In the proof progress, one knows that Theorem \ref{variety-iso-2} only requires $\sf_0$ is principal nilpotent in $\sg$. But we focus on the case that all $\sf_j$ are principal nilpotent for two aims in the future. One aim is to associate $\ssf$ to a nonsingular character of $\sn_l$. Another aim is to apply it to the manifold $Z$ of the normalized Jacobi elements.
\end{remark}

\subsection{}

The stability of $\ssf+\sb_l$ under $N_l$ induces an action of $N_l$ on $\sb_l$, which is defined by
\beqs a\cdot x=a(\ssf+x)-\ssf,\;\forall a\in N_l, x\in\sb_l.\eeqs
By Theorem \ref{variety-iso-2}, the map
\beq\label{variety-times-iso} N_l\times\sss_l\rightarrow\sb_l,\;(a,x)\mapsto a\cdot x\eeq
is an isomorphism. Noting that $a\cdot(a'\cdot x)=a\cdot(a'(\ssf+x)-\ssf)=a(a'(\ssf+x))-\ssf=(aa')\cdot x$, then $N_l$ acts as a group of automorphisms of the affine variety $\sb_l$.
Since $S(\sbo_l)$ is the affine algebra of $\sb_l$ this action of $N_l$ induces an $N_l$-module structure on $S(\sbo_l)$, the induced action is given as
\beqs (a\cdot u)(x)=u(a^{-1}\cdot x),\;\forall a\in N_l, u\in S(\sbo_l), x\in\sb_l.\eeqs
Of course $N_l$ operates as a group of automorphisms of $S(\sbo_l)$ and hence the space $S(\sbo_l)^{N_l}$ of $N_l$-invariants in $S(\sbo_l)$ is a subalgebra of $S(\sbo_l)$.

For any $u\in S(\sg_l)$ let $u^{\ssf}\in S(\sbo_l)$ defined by
\beqs u^{\ssf}(x)=u(\ssf+x),\;\forall x\in\sbo_l.\eeqs
Then there exists an algebra homomorphism
\beqs\rho_{\ssf}: S(\sg_l)\rightarrow S(\sbo_l),\;u\mapsto u^{\ssf}.\eeqs

 \begin{lemma}\label{invar-iso-1}For any $I\in S(\sg_l)^{G_l}$ one has $I^{\ssf}\in S(\sbo_l)^{N_l}$.\end{lemma}
\begin{proof} For any $I\in S(\sg_l)^{G_l}, x\in\sb_l$ and $a\in N_l$ one has
\beqs (a\cdot I^{\ssf})(x)=I^{\ssf}(a^{-1}\cdot x)=I(a^{-1}({\ssf}+x))=(aI)({\ssf}+x)=I({\ssf}+x)=I^{\ssf}(x),\eeqs
it follows $I^{\ssf}\in S(\sbo_l)^{N_l}$.
\end{proof}

\subsection{} For any $x\in\sg_l$ the extension of ${\rm ad}x$ to a derivation of $S(\sg_l)$ is still denoted by ${\rm ad}x$. In particular, the restriction of ${\rm ad}(-x_0)$ to $\sbo_l$ has non-negative integral eigenvalues and hence the same is true of its restriction to $S(\sbo_l)$.  Let $(S_k(\sbo_l))_j$ be the eigenspace in $S_k(\sbo_l)$ for ${\rm ad}(-x_0)$ corresponding to the eigenvalue $j$ and put
\beq\label{direct-sum-S} S_{(i)}(\sbo_l)=\bigoplus_{k+j=i}(S_k(\sbo_l))_j.\eeq
It is clear  that $S_{(i)}(\sbo_l)S_{(j)}(\sbo_l)\subseteq S_{(i+j)}(\sbo_l)$ and hence the $S_{(i)}(\sbo_l)$ define the structure of a graded algebra on $S(\sbo_l)$. It is clearly that $S_{(0)}(\sbo_l)=\mc1$, where $1$ is the identity of $S(\sbo_l)$.

\begin{theo}\label{invar-iso-graded}
One  has that $S(\sbo_l)^{N_l}$ is a graded subalgebra of $S(\sbo_l)$ with respect to the $x_0$-grading. In fact for any $k$ one has
\beqs\rho_\ssf: S_k(\sg_l)^{G_l}\rightarrow S_{(k)}(\sbo_l)^{N_l}.\eeqs
\end{theo}
\begin{proof}
Let $H(0)\subseteq G_l$ be the subgroup associated to $\sh(0)$ and let $S(\sg_l)^{H(0)}$ be the algebra of $H(0)$-invariant in $S(\sg_l)$. Clearly one has $S(\sg_l)^{G_l}\subseteq S(\sg_l)^{H(0)}$. By Lemma \ref{invar-iso-1} it suffices to show that $\rho_\ssf$ maps $S_k(\sg_l)^{H(0)}$ into $S_{(k)}(\sbo_l)$.

For any matrices $p=(p_{i,j})_{J\times(l+1)}, q=(q_{i,j})_{J\times(l+1)}$ and $m=(m_{i,j})_{n\times(l+1)}$, let $|p|=\sum_{i,j}p_{i,j}, |q|=\sum_{i,j}q_{i,j}, |m|=\sum_{i,j}m_{i,j}$ and
$\langle p\rangle=\sum_{i,j}p_{i,j}\beta_i, \langle q\rangle=\sum_{i,j}q_{i,j}\beta_i$. We define the following elements
\beqs e(p,m,q)=\left(\prod_{i,j}e_{-\beta_i}(j)^{p_{i,j}}\right)\left(\prod_{i,j}\omega_i(j)^{m_{i,j}}\right)\left(\prod_{i,j}e_{\beta_i}(j)^{q_{i,j}}\right).\eeqs
Clearly $S_k(\sg_l)^{H(0)}$ has a basis $\{e(p,m,q)\mid \langle p\rangle=\langle q\rangle\}$.

First one has $e(p,m,0)\in (S_{|p|+|m|}(\sbo))_j$ where $j=\langle p\rangle(x_0)$. Thus $e(p,m,0)\in S_{(t)}(\sbo)$ where $t=|p|+|m|+\langle p\rangle(x_0)$. If $q_{i,j}>0$ for some pair $(i,j)$ with $i>n$ then one has
\beqs e(p,m,q)^{\ssf}=0\eeqs
and the result holds for this case. If $q_{i,j}=0$ for all $i>n$, then $\langle q\rangle(x_0)=|q|$ and
\beqs e(p,m,q)^{\ssf}=\left(\prod_{i=1}^n\prod_{j=0}^l(Q(e_{\alpha_i}(l-j), e_{-\alpha_i}(j)))^{q_{i,j}}\right)e(p,m,0)\in S_{(t)}(\sbo_l).\eeqs
Noting that $\langle p\rangle=\langle q\rangle$, then one has $t=|p|+|m|+|q|=k$ and hence
\beqs e(p,m,q)^{\ssf}\in S_{(k)}(\sbo_l).\eeqs
The proof is finished.\end{proof}

\subsection{} Now we can assert the isomorphism between $S(\sg_l)^{G_l}$ and $S(\sbo_l)^{N_l}$.

 \begin{theo}\label{invar-iso-2}For any $I\in S(\sg_l)^{G_l}$ one has $I^{\ssf}\in S(\sbo_l)^{N_l}$. Moreover the map
 \beq\label{S(g)=iso=S(b)-2} S(\sg_l)^{G_l}\rightarrow S(\sbo_l)^{N_l},\;I\mapsto I^{\ssf}\eeq
 is an algebra isomorphism. In particular,
 \beqs S(\sbo_l)^{N_l}=\mc[{I_{i,j}}^{\ssf}: 1\leq i\leq n, 0\leq j\leq l]\eeqs
is a polynomial algebra in the ${I_{i,j}}^{\ssf}$. \end{theo}
\begin{proof} By Lemma \ref{invar-iso-1}, the map is well defined. By the isomorphism \eqref{variety-times-iso}, the restriction
\beqs S(\sbo_l)^{N_l}\rightarrow S({\sss_l}^*),\;v\mapsto v|_{\sss_l}\eeqs
is an algebra isomorphism.

Now we consider the restricted functions ${I_{i,j}}^\ssf|_{\sss_l}$ corresponding to the $G_l$-invariants $I_{i,j}$. Noting that $\sss\subseteq\sn$. Then one has
\beqs I_i(z)^\ssf|_{\sss_l}&=&\sum_{\gamma\in\Delta_+\atop \gamma=\alpha_{i_1}+\cdots+\alpha_{i_{m_i}}} c_\gamma (\overline{e_{-\gamma}}\cdot\overline{e_{\alpha_{i_1}}}\cdot...\cdot\overline{e_{\alpha_{i_{m_i}}}})^\ssf|_{\sss_l}\\
&&+\sum_{t>1, \gamma_j\in\Delta_+\atop \gamma_1+\cdots+\gamma_t=\alpha_{i_1}+\cdots+\alpha_{i_{m_i+1-t}}} c_{\gamma_1,\cdots,\gamma_t} (\overline{e_{-\gamma_1}}\cdot...\overline{e_{-\gamma_t}}\cdot\overline{e_{\alpha_{i_1}}}\cdot...\cdot\overline{e_{\alpha_{i_{m_i+1-t}}}})^\ssf|_{\sss_l}\eeqs
for some constants $c_{\gamma}, c_{\gamma_1,\cdots,\gamma_t}$.

 Let $G(0)\subseteq G_l$ be the subgroup corresponding to $\sg(0)$.  Clearly one has
\beqs \mc[I_{1,0},\cdots, I_{n,0}]=S(\sg(0)^*)^{G(0)}\eeqs
and ${I_{i,0}}^\ssf|_{\sss_l}={I_{i,0}}^\ssf|_{\sss(0)}$. By Theorem 7 in \cite{Kos2} the functions ${I_{i,0}}^\ssf|_{\sss(0)}$ is a coordinate system on $\sss(0)$. This implies that
\beqs s_{1,0},\cdots,s_{n,0}\eeqs
is a coordinate system on $\sss(0)$, where
\beqs s_{i,0}&=&\sum_{\gamma\in\Delta_+\atop \gamma=\alpha_{i_1}+\cdots+\alpha_{i_{m_i}}} c_\gamma ({e_{-\gamma}}(l)\cdot {e_{\alpha_{i_1}}(l)}\cdot...\cdot{e_{\alpha_{i_{m_i}}}(l)})^\ssf|_{\sss(0)}\\
&=&\sum_{\gamma\in\Delta_+\atop \gamma=\alpha_{i_1}+\cdots+\alpha_{i_{m_i}}} c_\gamma\left(\prod_{r=1}^{m_i}\kappa(\sf_0,e_{\alpha_{i_r}})\right){e_{-\gamma}}(l)|_{\sss(0)}
\eeqs

 In general for $k\geq 1$ we assume that the functions
\beqs {I_{i,j}}^{\ssf}|_{\sss_l}={I_{i,j}}^{\ssf}|_{\sss[k-1]}, 1\leq i\leq n, 0\leq j\leq k-1\eeqs
is a coordinate system on $\sss[k-1]$. By the definition  one has
\beqs I_{i,k}\in \sg(k)^*S_{m_i-1}(\sg(0)^*)\oplus S(\sg[k-1]^*),\eeqs
thus ${I_{i,k}}^{\ssf}|_{\sss_l}\in{I_{i,k}}^{\ssf}|_{\sss(k)}+S(\sss[k-1]^*)$. Explicitly, one has
\beqs {I_{i,k}}^{\ssf}|_{\sss(k)}&=&\sum_{\gamma\in\Delta_+\atop \gamma=\alpha_{i_1}+\cdots+\alpha_{i_{m_i}}} c_\gamma\left(\prod_{r=1}^{m_i}\kappa(\sf_0,e_{\alpha_{i_r}})\right){e_{-\gamma}}(l-k)|_{\sss(k)}.\eeqs
and hence
\beqs {I_{i,k}}^{\ssf}(x(k))={s_{i,0}}(x(0))\eeqs
for all $x\in\sss$. Thus the functions ${I_{i,k}}^{\ssf}|_{\sss(k)}, 1\leq i\leq n$ is a coordinate system on $\sss(k)$. Furthermore, the functions
\beqs {I_{i,j}}^{\ssf}|_{\sss_l}={I_{i,j}}^{\ssf}|_{\sss[k]}, 1\leq i\leq n, 0\leq j\leq k\eeqs
is a coordinate system on $\sss[k]$. By induction on $k$, one knows that the functions
\beqs {I_{i,j}}^{\ssf}|_{\sss_l}, 1\leq i\leq n, 0\leq j\leq l\eeqs
is a coordinate system on $\sss_l$.  This implies that the map  (\eqref{S(g)=iso=S(b)-2}) is surjective. Because the functions ${I_{i,j}}^{\ssf}$ are algebraically independent, the map (\eqref{S(g)=iso=S(b)-2}) is necessarily injective. Thus it is an isomorphism.
\end{proof}

\begin{theo}\label{S(g)-G-S(b)-N}
For any $k$ the map
\beqs\rho_\ssf: S_k(\sg_l)^{G_l}\rightarrow S_{(k)}(\sbo_l)^{N_l}\eeqs
is an isomorphism.\end{theo}
\begin{proof}It follows from Theorem \ref{invar-iso-graded} and Theorem \ref{invar-iso-2}.
\end{proof}

\subsection{}
Let $H_l, N_l$ and $B_l$ be the Lie subgroups of $G_l$ corresponding to Lie subalgebras $\sh_l, \sn_l$ and $\sb_l$, respectively. Then $H_l$ is abelian and $B_l$ is the semi-direct product $H_lN_l$. Now define an action of $B_l$ on $N_l$ by putting
\beqs b\cdot a=hnah^{-1}\eeqs
for all $b=hn\in B_l, a\in N_l$, where $h\in H_l, n\in N_l$.

Now $N_l$ is a unipotent algebraic group. Let $A(N_l)$ be the affine algebra of all regular functions on $N_l$. Then $B_l$ operates as a group of automorphisms on $A(N)$. Explicitly the action is defined by
\beqs (b\cdot \gamma)(a)=\gamma(b^{-1}\cdot a),\;\forall a\in B_l, \gamma\in A(N_l), a\in N_l. \eeqs

Noting that $A(N_l)\otimes S({\sss_l}^*)$ is the affine algebra of $N_l\times\sss_l$, by the isomorphism $N_l\times\sss_l\rightarrow\sb_l$ one has an algebra isomorphism
\beqs A(N_l)\otimes S({\sss_l}^*)\rightarrow S(\sbo_l).\eeqs

In particular for any element $\gamma\in A(N_l)$ there exists an element $v_\gamma\in S(\sbo_l)$ so that $v_\gamma(a\cdot x)=\gamma(a)$ for any $a\in N_l$ and $x\in\sss_l$.
Let $A\subseteq S(\sbo_l)$ be the space of all $v_\gamma, \gamma\in A(N_l)$, then $A$ is a subalgebra of $S(\sbo_l)$ and the map
\beqs A(N_l)\rightarrow A,\;\gamma\mapsto v_\gamma\eeqs
an isomorphism of algebras.

Noting that both $A(N_l)$ and $S(\sbo_l)$ are $N_l$-modules.

\begin{prop}\label{A-is-N-submodule-of-S(b)} $A$ is an $N_l$-submodule of $S(\sbo_l)$ and the map
\beq\label{A-A-n-mod-iso} A(N_l)\rightarrow A,\;\gamma\mapsto v_\gamma\eeq
is an isomorphism of $N_l$-modules.
\end{prop}
\begin{proof} For any $a,b\in N_l, \gamma\in A(N_l)$ and $s\in\sss_l$, one has
\beqs (a\cdot v_\gamma)(b\cdot x)=v_\gamma(a^{-1}\cdot (b\cdot x))=v_\gamma((a^{-1}b)\cdot x)=\gamma(a^{-1}b)=\gamma(a^{-1}\cdot b)=(a\cdot\gamma)(b),\eeqs
and hence $a\cdot v_\gamma=v_{a\cdot\gamma}$.
\end{proof}

We consider the one-parameter group $r(t)$ by putting
\beqs r(t)z=\exp(t(1+{\rm ad}\ssx))z,\;\forall z\in \sg_l.\eeqs
Clearly it leaves $\sb_l$ stable and $e^{-t}r(t)\in H_l$. For any $v\in S(\sbo_l)$ and $y\in\sb_l$ one lets $r(t)\cdot v\in S(\sbo_l)$ be defined by
\beqs (r(t)\cdot v)(y)=v(r(-t)y).\eeqs

\begin{prop}For any $k\in\mz_+$ and $v\in S_{(k)}(\sbo_l)$ one has
\beqs r(t)\cdot v=e^{-kt}v.\eeqs
Moreover a subspace $V\subseteq S(\sbo_l)$ is graded with respect to the $\ssx$-grading if and only if it is stable under the action of the one-parameter group $r(t)$.
\end{prop}
\begin{proof}
First  one has ${\sg_{[-j]}}_l=S_{(j+1)}(\sbo_l)\cap\sbo_l$ for any $j\geq0$. For any $x\in {\sg_{[-j]}}_l$ and $y\in\sb_l$ it holds
\beqs (r(t)\cdot x)(y)&=&Q(x,r(-t)y)=e^{-t}Q(x,\exp(-t{\rm ad}\ssx)y)\\
&=&e^{-t}Q(\exp(t{\rm ad}\ssx)x,y)=e^{-(j+1)t}Q(x,y)=(e^{-(j+1)t}x)(y).\eeqs
Then $r(t)\cdot x=e^{-(j+1)t}x$ and hence the proposition holds for $v\in\sbo_l$. Since $S(\sbo_l)$ is clearly generated by $\sbo_l$ and $r(t)$ operates as an automorphism, one has $r(t)\cdot v=e^{-kt}v$ for any
$v\in S_{(k)}(\sbo_l)$. This implies that any graded subspace $V$ is stable under the action of $r(t)$. Conversely we assume that $V$ is stable under $r(t)$ for all $t$. Then for any $v\in V$ there exist finitely many
elements $v_i\in S_{(i)}(\sbo_l)$ such that $v=\sum_i v_i$.  Clearly one has
\beqs r(t)\cdot v=\sum_{i}e^{-it}v_i\in V,\eeqs
so all components $v_i$ are necessarily in $V$. Thus $V$ is graded.
\end{proof}

\begin{lemma}For any $a\in N_l, x\in\sb_l$ and $t\in\mr$ one has
\beqs r(t)(a\cdot x)=(\exp(t{\rm ad}\ssx)\cdot a)\cdot r(t)x.\eeqs
\end{lemma}
\begin{proof}
Noting that $\ssf\in(\sg_l)_{[-1]}$, then  one has $r(t)\ssf=\ssf$. Thus
\beqs r(t)(a\cdot x)&=&r(t)(a(\ssf+x)-\ssf)=r(t)a(\ssf+x)-\ssf\\
&=&(\exp(t{\rm ad}\ssx)a\exp(-t{\rm ad}\ssx))r(t)(\ssf+x)-\ssf\\
&=&(\exp(t{\rm ad}\ssx)\cdot a)(\ssf+r(t)x)-\ssf\\
&=&(\exp(t{\rm ad}\ssx)\cdot a)\cdot r(t)x.\eeqs
\end{proof}

\begin{theo}\label{atimessbn}
The algebra $A$ is graded with respect to the $\ssx$-grading. In fact for any $\gamma\in A(N_l)$ one has
\beqs r(t)\cdot v_\gamma=v_{\exp(t{\rm ad}x_0)\cdot\gamma}.\eeqs

If tensor product maps to multiplication one has an isomorphism
\beq\label{S(b)=A-times-SG} A\otimes S(\sbo_l)^{N_l}\rightarrow S(\sbo_l)\eeq
of graded algebras with respect to the $\ssx$-grading. In particularly one has
\beqs a\cdot(v_\gamma I^{\ssf})=v_{a\cdot\gamma}I^{\ssf}\eeqs
for any $a\in N_l, \gamma\in A(N_l)$ and $I\in S(\sg_l)^{G_l}$.
\end{theo}
\begin{proof}Noting that $\sss_l$ is ${\rm ad}\ssx$-stable and hence also $r(t)$-stable. For any $a\in N_l, x\in\sss_l$ one has
\beqs (r(t)\cdot v_\gamma)(a\cdot x)&=&v_\gamma(r(-t)(a\cdot x))=v_\gamma((\exp(-t{\rm ad}\ssx)\cdot a)\cdot r(-t)x)\\
&=&\gamma((\exp(-t{\rm ad}\ssx)\cdot a)=(\exp(t{\rm ad}\ssx)\cdot \gamma)(a)\\
&=&v_{\exp(t{\rm ad}\ssx)\cdot \gamma}(a\cdot x),\eeqs
and hence
\beqs r(t)\cdot v_\gamma=v_{\exp(t{\rm ad}\ssx)\cdot\gamma}.\eeqs

By Proposition 3.5, $S(\sbo_l)^{N_l}$ is graded. Clearly $S(\sss_l^*)$ maps into $S(\sbo_l)^{N_l}$ under the isomorphism $A(N_l)\otimes S(\sss_l^*)\rightarrow S(\sbo_l)$. Then the isomorphism $A(N_l)\rightarrow A$ establishes the isomorphism $A\otimes S(\sbo_l)^{N_l}\rightarrow S(\sbo_l)$.

Since $N_l$ operates as a group of automorphisms of $S(\sbo_l)$ and $a\cdot I^{\ssf}=(a\cdot I)^{\ssf}=I^{\ssf}$, clearly one has $a\cdot(v_\gamma I^\ssf)=v_{a\cdot\gamma}I^\ssf$.
\end{proof}

\subsection{}

Noting that the action of $N_l$ on $S(\sbo_l)$ arises from an affine action of the unipotent group $N_l$ on $\sb_l$, thus $N_l\cdot v$ spans a finite-dimensional space for any $v\in S(\sbo_l)$. Then one obtains a representation on $\sn_l$ as a Lie algebra of $S(\sbo_l)$ which is given by
\beqs z\cdot v=\frac{d}{dt}(\exp(t{\rm ad}z)\cdot v)\mid_{t=0},\;\forall z\in\sn_l, v\in S(\sbo_l).\eeqs
\begin{remark}\label{N-finiteness}
It is easy to see that the action of $N_l$ on $S(\sbo_l)$ is locally finite. Then $v\in S(\sbo_l)^{N_l}$ if and only if $z\cdot v=0$ for all $z\in\sn_l$. Also $A$ is stable under the action of $N_l$ if and only if it is stable under the adjoint action of $\sn_l$. \end{remark}

Let $P:\sg_l\rightarrow\sbo_l$ be the projection map such that ${\rm Ker}P=\sn_l$.

\begin{prop}\label{adei-dot-1}For any $x\in\sbo_l, z\in\sn_l$ one has $z\cdot x=P[z,x]+Q([z,x],\ssf)1$ where $1$ is the identity of $S(\sbo_l)$.
\end{prop}
\begin{proof}By the definition one has
\beqs &&(z\cdot x)(y)=\frac{d}{dt}\left((\exp(t{\rm ad}z)\cdot x)(y)\right)\mid_{t=0}=\frac{d}{dt}x\left((\exp(-t{\rm ad}z)\cdot y\right)\mid_{t=0}\\
&&=\frac{d}{dt}x\left((\exp(-t{\rm ad}z)(\ssf+y)-\ssf\right)\mid_{t=0}=x(-[z,\ssf+y])1=Q([z,x],y+\ssf)1.\eeqs
Thus $z\cdot x=P[z,x]+Q([z,x],\ssf)1$.\end{proof}
If $z=e_{\alpha_i}(j)$ and $v\in S(\sbo_l)$, then there exist uniquely $v', v_{j'}\in S(\sbo_l)$ such that
\beq\label{adei-v}\ad e_{\alpha_i}(j)(v)=v'+\sum_{j'\geq j}v_{j'}e_{\alpha_i}(j'). \eeq

\begin{prop}\label{adei-v-prop}Let $i=1,\cdots,n, j=0,\cdots,l$ and let $v\in S(\sbo_l)$. Then
\beq\label{adei-v-equation} e_{\alpha_i}(j)\cdot v=v'+\sum_{j'\geq j}Q(e_{\alpha_i}(j'), e_{-\alpha_i}(l-j'))v_{j'}\eeq
for uniquely determined elements $v',v_{j'}\in S(\sbo_l)$ in \eqref{adei-v}.
\end{prop}
\begin{proof}When $v\in\sbo_l$ the statement holds by \eqref{adei-v} and Proposition \ref{adei-dot-1}.  Noting that $\ad e_{\alpha_i}(j)$ is a derivation of $S(\sg_l)$ and
\beqs (\ad e_{\alpha_i}(j))(uv)=(u'v+uv')+\sum_{j'\geq j}(u_{j'}v+uv_{j'})e_{\alpha_i}(j'),\eeqs
this implies
\beqs e_{\alpha_i}(j)\cdot(uv)=(u'v+uv')\sum_{j'\geq j}(u_{j'}v+uv_{j'})Q(e_{\alpha_i}(j'),e_{\alpha_i}(l-j'))\eeqs
and hence $e_{\alpha_i}(j)$ operates as a derivation of $S(\sbo_l)$. This proves the statement.
\end{proof}
\begin{prop}\label{z-dot-v}Let $z\in{\sg_{[j]}}_l, j\geq1$ and $v\in S_{(k)}(\sbo_l)$. Then $z\cdot v\in S_{(k-j)}(\sbo_l)$.
\end{prop}
\begin{proof}Since $\sn_l$ is generated by the elements $e_{\alpha_i}(j), 1\leq i\leq n, 0\leq j\leq l$, it suffices to assume $z=e_{\alpha_i}(j)$. Since $S_p(\sg_l)$ is preserved by all $\ad e_{\alpha_i}(j)$ and $[\ssx, e_{\alpha_i}(j)]=e_{\alpha_i}(j)$, then it follows $v'\in(S_p(\sbo_l))_{q-1}$ and $v_{j'}\in(S_{p-1}(\sbo_l))_{q}$ for $v\in (S_{p}(\sbo_l))_q$ where $p+q=k$ and the elements  $v', v_{j'}$ are given by \eqref{adei-v-equation}. So clearly $e_{\alpha_i}(j)\cdot v\in S_{(k-1)}(\sbo_l)$.
\end{proof}

\begin{prop}Let $z\in{\sg_{[j]}}_l, j\geq1$ and $v\in A_{(k)}$. Then $z\cdot v\in A_{(k-j)}$. Moreover, $z\cdot v=0$ for all $z\in\sn_l$ if and only if $v\in\mc1$.
\end{prop}
\begin{proof}Since $A$ is $\sn_l$-stable by Remark \ref{N-finiteness} then the first statement follows from Proposition \ref{z-dot-v}. The second statement follows from Remark \ref{N-finiteness} and the fact
\beqs A\cap S(\sbo_l)^{N_l}=\mc1.\eeqs
\end{proof}

\section{Decomposition of $U(\sbo_l)$}

Let $U(\sbo_l)$ be the universal enveloping algebra of the subalgebra $\sbo_l$. In this section we construct the Whittaker model $W(\sbo_l)$ of the center $Z(\sg_l)$ and then decompose $U(\sbo_l)$ into a tensor product of $\widetilde{A}\otimes W(\sbo_l)$. The frame and the ideal of this section
is due to Kostant which appeared in the study of Whittaker modules of $\sg$ in   \cite{Kos3}. However we still write all detailed proofs for the readability.

\subsection{}
For any Lie subalgebra $\frak{a}\subseteq\sg_l$ let $U(\frak{a})$ denote the universal enveloping algebra of $\frak{a}$. Of course $U(\frak{a})$ is regarded as a subalgebra of $U(\sg_l)$ in the usual sense.
For any $k\in\mathbb{N}$, let $U_k(\frak{a})$ denote the subspace spanned by all elements $x_1\cdots x_i, i\leq k, x_j\in\frak{a}$. Then one obtains  the well known filtration of $U(\frak{a})$:
\beqs U_0(\frak{a})\subseteq U_1(\frak{a})\subseteq\cdots\subseteq U_k(\frak{a})\subseteq U_{k+1}(\frak{a})\subseteq\cdots.\eeqs
In particular, $U_0(\frak{a})=\mc, U_1(\frak{a})=\mc+\frak{a}$.

Let $u(p,m,q)\in U(\sg)$ be the element with the same expression of $e(p,m,q)\in S(\sg)$. By the PBW Theorem, all such elements form a basis of $U(\sg)$. In particular, $U_k(\sg)$ is spanned by the elements
\beqs u(p,m,q),\; |p|+|m|+|q|\leq k.\eeqs

Clearly $U_k(\sbo_l)$ is stable under the action of $\ad\ssx$. For any $j\in\mathbb{N}$ let $(U_k(\sbo_l))_j$ denote the eigenspace of $U_k(\sbo_l)$ for $\ad(-\ssx)$ corresponding to the eigenvalue $j$ and put
\beq\label{directsum-U} U_{(i)}(\sbo_l)=\sum_{k+j=i}(U_k(\sbo_l))_j.\eeq
In particular, the subspaces $U_{(i)}(\sbo_l)$ also define a filtration of $U(\sbo_l)$, which is called the $\ssx$-filtration of $U(\sbo_l)$.

By the PBW theorem there is a unique series of exact sequences
\beq\label{exact-sequence-tau-k} \xymatrix@C=0.5cm{
  0 \ar[r] & U_{k-1}(\sg_l) \ar[rr]^{\rm inj} && U_k(\sg_l) \ar[rr]^{\tau_{k}} && S_k(\sg_l) \ar[r] & 0 }
\eeq
for all $k\in\mz_+$ such that $\tau_1|_{\sg_l}$ is the identity map and
\beqs \tau_{i+j}(uv)=\tau_i(u)\tau_j(v)\eeqs
for all $u\in U_i(\sg_l), v\in U_j(\sg_l)$. If $\sg_l$ is replaced by $\sbo_l$ or any other subalgebra, clearly the  exact sequences are still maintained. Moreover $\tau_k$ commutes with the adjoint action of $\sg_l$ and one obtains an induced exact sequence
\beqs \xymatrix@C=0.5cm{
  0 \ar[r] & (U_{k-1}(\sbo_l))_j \ar[rr]^{\rm inj} && (U_k(\sbo_l))_j \ar[rr]^{\tau_{k}} && (S_k(\sbo_l))_j \ar[r] & 0 }
\eeqs

Noting that \eqref{directsum-U} is a direct sum. Any $u\in U_{(i)}(\sbo_l)$ can be written as $u=\sum_{j}u_j$ where $u_j\in(U_{i-j}(\sbo_l))_j$ is uniquely determined. So there exists a well-defined map
\beqs \tau_{(i)}:U_{(i)}(\sbo_l)\rightarrow S_{(i)}(\sbo_l),\;u\mapsto\sum_{j}\tau_{i-j}(u_j).\eeqs

\begin{prop}\label{exact-sequence-tau(i)-prop}For any $i\in\mz_+$ one has an exact sequence
\beq\label{exact-sequence-tau(i)} \xymatrix@C=0.5cm{
  0 \ar[r] & U_{(i-1)}(\sbo_l) \ar[rr]^{\rm inj} && U_{(i)}(\sbo_l) \ar[rr]^{\tau_{(i)}} && S_{(i)}(\sbo_l) \ar[r] & 0 }
\eeq
Moreover,
\beq\label{tau-grading} \tau_{(i+j)}(uv)=\tau_{(i)}(u)\tau_{(j)}(v)\eeq
for all $u\in U_{(i)}(\sbo_l), v\in U_{(j)}(\sbo_l)$.
\end{prop}
\begin{proof}Noting that both sums \eqref{direct-sum-S} and \eqref{directsum-U} are direct, then \eqref{exact-sequence-tau(i)} follows from \eqref{exact-sequence-tau-k}.

Let $u_p\in (U_{i-p})_p, v_q\in (U_{i-q})_q$ be such that $u=\sum_pu_p$ and $v=\sum_qv_q$. Then one has
\beqs \tau_{(i+j)}(uv)=\sum_r\tau_{i+j-r}\left(\sum_{p+q=r}u_pv_q\right)=\sum_r\sum_{p+q=r}\tau_{i-p}(u_p)\tau_{j-q}(v_q)=\tau_{(i)}(u)\tau_{(j)}(v),\eeqs
this finishes the proof.
\end{proof}

 Let $\gr U(\sbo_l)$ be the graded algebra associated to the $\ssx$-filtration. Then
\beqs \gr U(\sbo_l)=\bigoplus_{i=0}^\infty\gr U_{(i)}(\sbo_l)\eeqs
where $\gr U_{(i)}(\sbo_l)=U_{(i)}(\sbo_l)/U_{(i-1)}(\sbo_l)$. Clearly $\gr U(\sbo_l)$ is a commutative algebra.  Now we define a map $\tau_{\ssx}:\gr U(\sbo_l)\rightarrow S(\sbo_l)$ so that its restriction to  $\gr U_{(i)}(\sbo_l)$ is the map
\beqs
U_{(i)}(\sbo_l)/U_{(i-1)}(\sbo_l)\rightarrow S_{(i)}(\sbo_l)
\eeqs
induced by $\tau_{(i)}$. Obviously by Proposition \ref{exact-sequence-tau(i)-prop} then one has
\begin{prop}The map $\tau_{\ssx}$ is an isomorphism of graded commutative algebras.
\end{prop}

\subsection{} Let $Z(\sg_l)$ be the center of $U(\sg_l)$. The standard filtration in $U(\sg_l)$ induces a filtration $Z_k(\sg_l)$ in $Z(\sg_l)$. To each invariant $I_{i,j}\in S_{m_i+1}(\sg_l)^{G_l}$ there exists
$\widetilde{I_{i,j}}\in Z(\sg_l)\cap U_{m_i+1}(\sg_l)$ such that $\tau_{m_i+1}(\widetilde{I_{i,j}})=I_{i,j}$. In particular, $Z(\sg_l)$ is a polynomial algebra in the $\widetilde{I_{i,j}}$ and one has an exact sequence
\beq\label{exact-sequence-Z-tau-k} \xymatrix@C=0.5cm{
  0 \ar[r] & Z_{k-1}(\sg_l) \ar[rr]^{\rm inj} && Z_{k}(\sg_l) \ar[rr]^{\tau_{k}} && S_{k}(\sg_l)^{G_l} \ar[r] & 0 }
\eeq

Let $\eta:\sn_l\rightarrow\mc$ be a homomorphism of Lie algebras, this is a one-dimensional representation of $\sn_l$. Then $\eta$ is uniquely determined by the constants $d_{i,j}=\eta(e_{\alpha_i}(j)), 1\leq i\leq n, 0\leq j\leq l$ and the constants $d_{i,j}$ are arbitrary. The representation $\eta$ is called nonsingular if $d_{i,j}\not=0$ for all $i,j$.

Extending $\eta$ to a homomorphism $\eta:U(\sn_l)\rightarrow\mc$ of algebras, let $U_\eta(\sn_l)$ denote the kernel of this extended $\eta$ so that one has a direct sum
\beq\label{decomposition-U(n)} U(\sn_l)=\mc1\oplus U_\eta(\sn_l).\eeq
Since $\sg_l=\sbo_l\oplus\sn_l$ there exists a linear isomorphism $U(\sg_l)\cong U(\sbo_l)\otimes U(\sn_l)$ and hence one has the direct sum
\beq\label{decomposition-U(g)} U(\sg_l)=U(\sbo_l)\oplus U(\sg_l)U_\eta(\sn_l).\eeq

For any $u\in U(\sg_l)$ let $u^\eta\in U(\sbo_l)$ be its component in $U(\sbo_l)$ relative to the above decomposition. Define a linear map
\beqs \rho_\eta: U(\sg_l)\rightarrow U(\sbo_l),\; u\mapsto u^\eta.\eeqs
Let $W(\sbo_l)=\rho_\eta(Z(\sg_l))$.

\begin{lemma}\label{lemma-Z-W-ep}The epimorphism
\beqs Z(\sg_l)\rightarrow W(\sbo_l),\; u\mapsto u^\eta\eeqs
is a homomorphism of algebras.\end{lemma}
\begin{proof}For any $u,v\in Z(\sg_l)$ one has
\beqs uv-u^\eta v^\eta=(u-u^\eta)v+u^\eta(v-v^\eta)=v(u-u^\eta)+u^\eta(v-v^\eta)\in U(\sg_l)U_\eta(\sn_l),\eeqs
which proves $(uv)^\eta=u^\eta v^\eta$.\end{proof}

Clearly the proof for Lemma \ref{lemma-Z-W-ep} is independent on the choice of $\eta$. For any $u\in Z(\sg_l)$, we write $u^0$ for $u^\eta$ when $\eta$ is trivial then $u^0\in U(\sh_l)$. In fact the map
\beqs Z(\sg_l)\rightarrow U(\sh_l),\;u\mapsto u^0\eeqs
is a generalized analogue of the well-known Harish-Chandra homomorphism $Z(\sg)\rightarrow U(\sh)$ (for example see (2.3.9) in \cite{Kos3}).

\begin{prop}\label{u-eta-u-zero}For any $u\in Z(\sg_l)$ one has
\beqs u^\eta-u^0\in\sn_l(\sbo_l).\eeqs
\end{prop}
\begin{proof}First one has the following
\beqs U(\sbo_l)&=&\sno_lU(\sbo_l)\oplus U(\sh_l),\\
U(\sg_l)&=&U(\sno_l)U(\sh_l)U(\sn_l).\eeqs
Then \eqref{decomposition-U(g)} implies
\beq\label{U(g)=U(b)+U(h)+n-eta} U(\sg_l)=U(\sbo_l)\oplus U(\sbo_l)U_\eta(\sn_l)=\sno_lU(\sbo_l)\oplus U(\sh_l)\oplus U(\sbo_l)U_\eta(\sn_l).\eeq
Let $v$ be the component of $u^\eta$ in $U(\sh_l)$, then $v$ is the component of $u$ in $U(\sh_l)$ relative to \eqref{U(g)=U(b)+U(h)+n-eta}. But one has $u-u^0\in\sno_lU(\sg_l)$ as well as $u-u^0\in U(\sg_l)\sn_l$. Thus $u^0$ is the component of $u$ in $U(\sh_l)$ relative to \eqref{U(g)=U(b)+U(h)+n-eta}, this is independent on $\eta$ and thus $v=u^0$.
\end{proof}

\subsection{}

Since the $\sf_j$ can be any principal nilpotent elements, now for given $\eta$ we will say that $\ssf=\sum_{j=0}^l\sf_j(j)$ corresponds to $\eta$ in the case
\beqs\eta(e_{\alpha_i}(j))=Q(\ssf,e_{\alpha_i}(j))=Q(e_{-\alpha_i}(l-j),e_{\alpha_i}(j)).\eeqs

Let $W_{(k)}(\sbo_l)$ be the $\ssx$-filtration induced from the $\ssx$-filtration in $U(\sbo_l)$. Then $\tau_{(k)}$ defines a map $W_{(k)}(\sbo_l)\rightarrow S_{(k)}(\sbo_l)$.

\begin{theo}\label{Z-W-grading-theorem}For any $k\in\mz_+$ the map $\rho_\eta$ induces an isomorphism
\beq \label{Z-W-grading} Z_k(\sg_l)\rightarrow W_{(k)}(\sbo_l).\eeq

If $\eta$ is nonsingular and the element $\ssf$ corresponds to $\eta$ then one has the following commutative diagram
\beq\label{two-line-exact}
\xymatrix@C=0.5cm{
  0 \ar[r] & Z_{k-1}(\sg_l)\ar[d]^{\rho_\eta}\ar[rr]^{{\rm inj}} && Z_k(\sg_l)\ar[d]^{\rho_\eta} \ar[rr]^{\tau_k} && S_k(\sg_l)^{G_l}\ar[d]^{\rho_\ssf} \ar[r] & 0 \\
  0 \ar[r] & W_{(k-1)}(\sbo_l) \ar[rr]^{{\rm inj}} && W_{(k)}(\sbo_l) \ar[rr]^{\tau_{(k)}} && S_{(k)}(\sbo_l)^{N_l} \ar[r] & 0 }
\eeq
where the horizontal lines are exact sequences and the vertical lines are isomorphisms.
\end{theo}
\begin{proof}Clearly the map $\rho_\eta: Z(\sg_l)\rightarrow W(\sbo_l)$ is bijective by \eqref{invariants-S(g-l)}, the maps $\tau_k$, the Harish-Chandra homomorphism of $\sg$ and Proposition \ref{u-eta-u-zero}.

Let $U(\sg_l)^{\sh(0)}$ be the centralizer of $\sh(0)$ in $U(\sg_l)$. Then $Z_k(\sg_l)\subseteq U_k(\sg_l)^{\sh(0)}$. The elements
\beqs\{u(p,m,q)|\langle p\rangle=\langle q\rangle, |p|+|m|+|q|\leq k\} \eeqs
form a basis of $U_k(\sg_l)^{\sh(0)}$. Given such an element it is clear that
\beqs u(p,m,q)^\eta=\left\{\begin{array}{ll}
\left(\displaystyle\prod_{i=1}^n\prod_{j=0}^{l}\eta(e_{\beta_i}(j))^{q_{i,j}}\right)u(p,m,0),&q_{i,j}=0, \forall i>n,\\
0,&{\rm otherwise}.
\end{array}\right.\eeqs
Then  it follows by $\langle p\rangle=\langle q\rangle$ that $u(p,m,q)^\eta\in (U_{|p|+|m|}(\sbo_l))_{|q|}=U_{(|p|+|m|+|q|)}(\sbo_l)\subseteq U_{(k)}(\sbo_l)$  in the case that $q_{i,j}=0$ for all $i>n$. Otherwise it is clear that $u(p,m,q)^\eta=0\in U_{(k)}(\sbo_l)$.  Thus
\beqs \rho_\eta(Z(\sg_l)_k)\subseteq U_{(k)}(\sbo_l)\cap W(\sbo_l)=W_{(k)}(\sbo_l).\eeqs
We have proved the first statement of this theorem. It is also clear that the first square in the diagram \eqref{two-line-exact} is commutative.

Now consider the following diagram
\beq\label{diagramsbquare} \xymatrix@C=0.5cm{
U_{k}(\sg_l)^{\sh(0)}\ar[d]^{\rho_\eta}\ar[rr]^{\tau_{k}} && S_{k}(\sg_l)\ar[d]^{\rho_{\ssf}}\\
U_{(k)}(\sbo_l)  \ar[rr]^{\tau_{(k)}} && S_{(k)}(\sbo_l)
}\eeq
For any basis element $u(p,m,q)$ of $U(\sg_l)^{\sh(0)}$ clearly one has $u(p,m,q)\in{\rm Ker}\tau_k\cap{\rm Ker}\rho_\eta$ whenever $|p|+|m|+|q|<k$.
Now assume $|p|+|m|+|q|=k$. Then one has
\beqs &&\rho_{\ssf}(\tau_k(u(p,m,q)))=e(p,m,q)^\ssf=\left(\displaystyle\prod_{i=1}^n\prod_{j=0}^{l}Q(e_{\beta_i}(j),e_{-\beta_i}(l-j))^{q_{i,j}}\right)e(p,m,0),\\
&&\tau_{(k)}(\rho_\eta(u(p,m,q)))=\left(\displaystyle\prod_{i=1}^n\prod_{j=0}^{l}\eta(e_{\beta_i}(j))^{q_{i,j}}\right)e(p,m,0),\eeqs 
and hence $\rho_{\ssf}(\tau_k(u(p,m,q)))=\tau_{(k)}(\rho_\eta(u(p,m,q)))$ by the assumption that $\ssf$ corresponds to $\eta$. The diagram (\ref{diagramsbquare}) is commutative. If we replace $U_k(\sg_l)^{\sh(0)}$ by $Z_k(\sg_l)$ and then replace $S_{k}(\sg_l), U_{(k)}(\sbo_l), S_{(k)}(\sbo_l)$ by
$S_k(\sg_l)^{G_l}, W_{(k)}(\sbo_l), S_{(k)}(\sbo_l)^{N_l}$ respectively, we get the second square of diagram (\ref{two-line-exact}). The whole diagram (\ref{two-line-exact}) is commutative.

Finally, the top horizontal line is exact by \eqref{exact-sequence-Z-tau-k} and the vertical lines are isomorphisms by \eqref{Z-W-grading} and Theorem \ref{S(g)-G-S(b)-N}. Then the second line is exact by the exactness of the first line and the vertical isomorphisms $\rho_\eta$ and $\rho_\ssf$.\end{proof}

\begin{theo}\label{whittaker-model-theorem}If $Z(\sg_l)$ is filtrated by the standard filtration and $W(\sbo_l)$ is filtrated by the $\ssx$-filtration. Then for any $\eta$ the map
\beqs \rho_\eta: Z(\sg_l)\rightarrow W(\sbo_l),\; u\mapsto u^\eta\eeqs
is an isomorphism of filtrated algebras. In particular,
\beq\label{whittaker-model}W(\sbo_l)=\mc[\widetilde{I_{i,j}}^\eta: 1\leq i\leq n, 0\leq j\leq l]\eeq
is a polynomial algebra in the generators $\widetilde{I_{i,j}}^\eta\in W_{(m_i+1)}(\sbo_l)$.\end{theo}
\begin{proof}It is a direct consequence of Theorem \ref{Z-W-grading-theorem} and Lemma \ref{lemma-Z-W-ep}.
\end{proof}
\begin{defi}[see Definition A in \cite{S}]
The   algebra  $W(\sbo_l)$ is called the {\it Whittaker model} of the center $Z(\sg_l)$ of the universal enveloping algebra $U(\sg_l)$ of $\sg_l$.
\end{defi}

\subsection{}

Let $\eta$ be a fixed character $\sn_l\rightarrow\mc$ and is nonsingular. Let $\ssf$ be chosen to correspond to $\eta$. Noting that $A$ is a graded subalgebra and by Theorem \ref{atimessbn} one can write $S(\sbo_l)=A\otimes S(\sbo_l)^{N_l}$. Now let $\widetilde{A_{(k)}}\subseteq U_{(k)}(\sbo_l)$ be any fixed subspace such that $\tau_{(k)}$ induces a linear isomorphism
\beq\label{aagraded} \widetilde{A_{(k)}}\rightarrow {A_{(k)}}.\eeq
Put $\widetilde{A}=\sum_{j=0}^\infty\widetilde{A_{(j)}}$. Clearly the sum is direct. Now $U(\sbo_l)$ is regarded as a right module over $W(\sbo_l)$ with respect to right multiplication.

\begin{theo}\label{U(b)-decom-A-W-homegeneous-theorem}$U(\sbo_l)$ is a free right module over $W(\sbo_l)$ and the multiplication induces an isomorphism
\beqs \widetilde{A}\otimes W(\sbo_l)\rightarrow U(\sbo_l)\eeqs
of right $W(\sbo_l)$-modules. Moreover this isomorphism induces an isomorphism
\beq\label{U(b)-decom-A-W-homegeneous} \bigoplus_{p+q=k}\widetilde{A_{(p)}}\otimes W_{(q)}(\sbo_l)\rightarrow U_{(k)}(\sbo_l)\eeq
for any $k\in\mathbb{N}$.
\end{theo}
\begin{proof}Clearly one has $\mc1=U_{(0)}(\sbo_l)=\widetilde{A}_{(0)}W_{(0)}(\sbo_l)$. Then inductively we can assume that \eqref{U(b)-decom-A-W-homegeneous} holds with replacing $k$ by any $j<k$. Let $u\in U_{(k)}(\sbo_l), v\in S_{(k)}(\sbo_l)$ be such that $\tau_{(k)}(u)=v$. By \eqref{S(b)=A-times-SG} we can uniquely find $v_i\in \widetilde{A}$ and $z_i\in S_{(k-i)}(\sbo_l)^{N_l}$ such that $v=\sum_iv_iz_i$. By \eqref{two-line-exact} we can find $w_i\in W_{(k-i)}(\sbo_l)$ such that $\tau_{(k-i)}(w_i)=z_i$. Let $u'$ be the image $\sum_iu_iw_i$ of
$\sum_iv_iz_i$ then one has $u-u'\in U_{(k-1)}(\sbo_l)$ and it appears in the image of \eqref{U(b)-decom-A-W-homegeneous} by inductive assumption.
Thus $u$ and hence $U(\sbo_l)$ is in the image of \eqref{U(b)-decom-A-W-homegeneous}. The map \eqref{U(b)-decom-A-W-homegeneous} is surjective.

The proof for injectivity is similar. For any nontrivial relation $\sum_iu_iw_i=0$, by considering its highest degree terms we have a nontrivial relation $\sum_iv_iz_i=0$ where $z_i\in S(\sbo_l)^{N_l}$.\end{proof}

\subsection{}

Let $\mc_\eta$ be the one-dimensional $\sn_l$-module defined by the character $\eta$ and let $W_\eta$ denote the induced module
\beqs W_\eta=U(\sg_l)\otimes_{U(\sn_l)}\mc_\eta,\eeqs
which is referred to the universal Whittaker module over $U(\sg_l)$ corresponding to $\eta$. Clearly it is the left quotient $U(\sg_l)/U(\sg_l)U_\eta(\sn_l)$ and hence
one has an epimorphism
\beqs U(\sg_l)\rightarrow W_\eta.\eeqs
Let $1_\eta$ denote the image of $1$. For any $u\in U(\sg_l), y\in W_\eta$ the action of $u$ on $y$ is denoted by $uy$. Clearly, one has
\beq\label{u-eta} u1_\eta=u^\eta1_\eta.\eeq
In particular, $W_\eta$ is a free $U(\sbo_l)$-module of rank one and this induces a $U(\sg_l)$-module structure on $U(\sbo_l)$ which is defined by
\beqs u\circ w^\eta=(uw)^\eta.\eeqs
Now one obtains an isomorphism
\beqs U(\sbo_l)\rightarrow W_\eta\eeqs
of $U(\sg_l)$-modules.

Next we consider the $\eta$-reduced action of $\sn_l$ on $U(\sbo_l)$:
\beqs x\cdot v=x\circ v-\eta(x)v.\eeqs

\begin{lemma}\label{adjoint-u-eta}For any $u\in U(\sg_l), x\in\sn_l$ one has
\beqs x\cdot u^\eta=[x,u]^\eta.\eeqs
\end{lemma}
\begin{proof}First one has
\beqs xu1_\eta=[x,u]1_\eta+ux1_\eta=[x,u]1_\eta+\eta(x)u1_\eta\eeqs
and
\beqs [x,u]1_\eta=xu1_\eta-\eta(x)u1_\eta.\eeqs
Therefore  it holds
\beqs (x\cdot u^\eta)1_\eta=(x\circ u^\eta-\eta(x)u^\eta)1_\eta=(xu-\eta(x)u)^\eta1_\eta=[x,u]^\eta1_\eta.\eeqs
Then  one has $x\cdot u^\eta=[x,u]^\eta$ by the uniqueness of $u^\eta$ in \eqref{u-eta}.
\end{proof}

\begin{theo}Let $p,q\in\mathbb{N}, p>0, x\in{\sg_{[p]}}_l\subseteq\sn_l$ and $u\in U_{(q)}(\sbo_l)$. Then $x\cdot u\in U_{(q-p)}(\sbo_l)$ so that the $\eta$-reduced action of $\sn_l$ induces a graded action of $\sn_l$ on $\gr U(\sbo_l)$. Moreover one has
\beq\label{tau-k-p-q} \tau_{q-p}(x\cdot u)=x\cdot\tau_{(q)}u\eeq
and hence the map
\beq\label{iso-tau-k} \tau_\ssx:\gr U(\sbo_l)\rightarrow S(\sbo_l),\eeq
is an isomorphism of graded $\sn_l$-modules.
\end{theo}
\begin{proof}
It is sufficient to assume $x=e_{\alpha_i}(j)$ for some $1\leq i\leq n, 0\leq j\leq l$ since  $\sn_l$ is generated by the elements $e_{\alpha_i}(j)$.

Because $[e_{\alpha_i}(j),U(\sno_l)]\subseteq U(\sbo_l)$ and $[e_{\alpha_i}(j),U(\sh_l)]\subseteq \sum_{j'\geq j}U(\sh_l)e_{\alpha_i}(j')$, then there  are unique elements $u', u_{j'}\in U(\sbo_l)$ with $j'\geq j$ for any $u$ such that
\beq\label{ade-u}[e_{\alpha_i}(j),u]=u'+\sum_{j'\geq j}u_{j'}e_{\alpha_i}(j).\eeq
Since $\ad e_{\alpha_i}(j)$ is a derivation of $U(\sg_l)$ this establishes \eqref{ade-u}. The uniqueness follows from the isomorphism $U(\sg_l)\cong U(\sbo_l)\otimes U(\sn_l)$.

Now taking $u\in (U_k(\sbo_l))_r$ with $k+r=q$. Then $u\in U_k(\sg_l), [\ssx, e_{\alpha_i}(j)]=e_{\alpha_i}(j)$ and $[\ssx,u]=-ru$, this follows
$[\ssx,[ e_{\alpha_i}(j),u]]=(1-r)[e_{\alpha_i}(j),u]$. By the uniqueness of \eqref{ade-u} and
$$U_k(\sg_l)=\sum_{t+s=k}U_t(\sbo_l)\otimes U_s(\sn_l),$$
it follows that $u'\in (U_k(\sbo_l))_{r-1}$ and $u_{j'}\in (U_{k-1}(\sbo_l))_{r}$. So $u', u_{j'}\in U_{(q-1)}(\sbo_l)$.  Since ${u'}^\eta=u'$ and
\beqs (u_{j'}e_{\alpha_i}(j'))^\eta=\eta(e_{\alpha_i}(j'))u_{j'}=Q(e_{\alpha_i}(j'),e_{-\alpha_i}(l-j'))u_{j'},\eeqs
one has
\beq\label{ei-dot-u} e_{\alpha_i}(j)\cdot u=u'+\sum_{j'\geq j}Q(e_{\alpha_i}(j'),e_{-\alpha_i}(l-j'))u_{j'}\eeq
and hence $e_{\alpha_i}(j)\cdot u\in U_{(q-1)}(\sbo_l)$. The first statement holds.

In order to prove \eqref{tau-k-p-q} again it suffices to assume $x=e_{\alpha_i}(j)$ with the same reason. Assume $u\in (U_k(\sbo_l))_r$ with $q=k+r$. By the definition of $\tau_{(q-1)}$ and \eqref{ei-dot-u} one has
\beqs\tau_{(q-1)}(e_{\alpha_i}(j)\cdot u)=\tau_k(u')+\sum_{j'\geq j}Q(e_{\alpha_i}(j'),e_{-\alpha_i}(l-j'))\tau_{k-1}(u_{j'}).\eeqs
Since $\tau_k$ commutes with the adjoint action, by \eqref{ade-u} and Proposition \ref{adei-v-prop} then one has
\beqs e_{\alpha_i}(j)\cdot(\tau_{q}(u))&=&[e_{\alpha_i}(j),\tau_{q}(u)]^\eta=\big(\tau_{q-1}([e_{\alpha_i}(j),\tau_{q}(u)]\big)^\eta\\
&=&\left(\tau_{(q-1)}(u'+\sum_{j'\geq j}u_{j'}e_{\alpha_i}(j))\right)^\eta\\
&=&\left(\tau_{(q-1)}(u')+\sum_{j'\geq j}\tau_{(q)}(u_{j'})e_{\alpha_i}(j')\right)^\eta\\
&=&\tau_k(u')+\sum_{j'\geq j}Q(e_{\alpha_i}(j'),e_{-\alpha_i}(l-j'))\tau_{k-1}(u_{j'})\\
&=&\tau_{(q-1)}(e_{\alpha_i}(j)\cdot u).\eeqs
This finishes the proof.\end{proof}

\begin{lemma}\label{n-dot-u-Z-lemma}For any $u\in U(\sbo_l), v\in W(\sbo_l)$ and $x\in\sn_l$ one has
\beq\label{n-dot-u-Z} x\cdot(uv)=(x\cdot u)v\eeq
and hence
\beq\label{n-dot-Z=0} x\cdot v=0.\eeq \end{lemma}
\begin{proof}
By Theorem \ref{whittaker-model-theorem} there exists a unique $w\in Z(\sg_l)$ such that $v=w^\eta$. Since $u\in U(\sbo_l)$ one has $u=u^\eta$ and  $uw-u^\eta w^\eta=u(w-w^\eta)\in U(\sg_l)U_\eta(\sn_l)$. Then  $uv=(uw)^\eta$ and then $x\cdot(uv)=[x,uv]^\eta$ by Lemma \ref{adjoint-u-eta}. However $[x,uw]=[x,u]w$ since $w$ is central. Similar to $uv=(uw)^\eta$ one has $([x,u]w)^\eta=[x,u]^\eta w^\eta$. Thus $x\cdot(uv)=[x,u]^\eta v$. However $[x,u]^\eta=x\cdot u$ by Lemma \ref{adjoint-u-eta}, this proves \eqref{n-dot-u-Z}.

In particular by takeing $u=1$ then one has $x\cdot v=[x,w]^\eta=0$, this proves \eqref{n-dot-Z=0}.
\end{proof}

\subsection{}

For any $j\in\mz_+, k\in\mathbb{N},v\in\widetilde{A_{(k)}}$ and $x\in\sn_l\cap{\sg_{[j]}}_l$ one has $x\cdot v\in U_{(k-j)}(\sbo_l)$ and it can be uniquely written
\beq\label{x-dot-v-decomposition} x\cdot v=\sum_{p\geq0} u_p\eeq
for some elements $u_p\in\widetilde{A_{(p)}}\otimes W_{(q)}(\sbo_l)$ where $p+q=k-j$. Now we let $x\star v= u_{k-j}$ so that
\beq\label{star-definition} x\star v\in\widetilde{A_{(k-j)}}.\eeq

\begin{lemma}\label{dot-tau-commute-lemma}Let $j\in\mz_+, k\in\mathbb{N}, v\in\widetilde{A_{(k)}}$ and $x\in\sn_l\cap{\sg_{[j]}}_l$ then one has
\beq\label{w=x-dot-v-x-star-v} x\cdot v-x\star v\in U_{(k-j-1)}(\sbo_l)\eeq
and
\beq\label{dot-tau-commute} x\cdot(\tau_{(k)}v)=\tau_{(k-j)}(x\star v).\eeq
\end{lemma}
\begin{proof}
By \eqref{tau-k-p-q} one has $x\cdot(\tau_{(k)}v)=\tau_{(k-j)}(x\cdot v)$. Then by \eqref{x-dot-v-decomposition} one has
$x\cdot(\tau_{(k)}v)=\sum_{p\geq0}v_p$ where $v_p=\tau_{(k-j)}u_p$. Then
\beqs v_p\in\sum_{p+q=k-j}A_{(p)}\otimes S_{(q)}(\sbo_l)^{N_l}\eeqs
by \eqref{tau-grading}. However $x\cdot(\tau_{(k)}v)\in A_{(k-j)}$ by Proposition \ref{A-is-N-submodule-of-S(b)} and \eqref{adei-v}. Then $v_p=0$ for all $p\not=k-j$ and $x\cdot(\tau_{(k)}v)=v_{k-j}=\tau_{(k-j)}u_{(k-j)}=\tau_{(k-j)}(x\star v)$. This proves \eqref{dot-tau-commute}.

Let $w=x\cdot v-x\star v$. Then one has
\beqs \tau_{(k-j)}(w)=\tau_{(k-j)}(x\cdot v)-\tau_{(k-j)}(x\star v)=x\cdot(\tau_{(k-j)}v)-\tau_{(k-j)}(x\star v)=0.\eeqs
This proves \eqref{w=x-dot-v-x-star-v}.\end{proof}

\begin{prop}\label{A-tilde-A-n-mod-iso}
$\widetilde{A}$ is an $\sn_l$-module with respect to the $\eta$-reduced action $x\star v$ for any $x\in\sn_l$ and $v\in\widetilde{A}$.  Moreover any linear isomorphism
\beqs \widetilde{A}\rightarrow A \eeqs
of graded vector spaces defined by \eqref{aagraded} is an isomorphism of $\sn_l$-modules.
\end{prop}
\begin{proof}
It is easy to see that the action of $N_l$ on $S(\sbo_l)$ is locally finite. Then $v\in S(\sbo_l)^{N_l}$ if and only if $z\cdot v=0$ for all $z\in\sn_l$. Also $A$ is stable under the action of $N_l$ if and only if it is stable under the adjoint action of $\sn_l$. Thus by Proposition \ref{A-is-N-submodule-of-S(b)} the map \eqref{A-A-n-mod-iso} is also an isomorphism of $\sn_l$-modules.
Then this proposition follows from \eqref{dot-tau-commute}.
\end{proof}

\begin{lemma}\label{ei-star-U(b)}For any $p\in\mathbb{N}, v\in\widetilde{A_{(p)}}$ one has $e_{\alpha_i}(j)\star v\in\widetilde{A}_{(p-1)}$ for any $1\leq i\leq n$ and $0\leq j\leq l$. If $p\geq 1$ and $v\not=0$ there exists $i, j$ such that $e_{\alpha_i}(j)\star v\not=0$. In this case one also has $e_{\alpha_i}(j)\cdot v\not=0$.
\end{lemma}
\begin{proof}
The first statement follows from \eqref{star-definition}. The second statement follows from Proposition \ref{A-tilde-A-n-mod-iso}. By the definition of $\widetilde{A_{(k)}}$ one has $\widetilde{A_{(k-j)}}\cap U_{(k-j-1)}(\sbo_l)=0$. By Lemma \ref{dot-tau-commute-lemma} one has $x\cdot v=0$ if and only if $x\star v=0$. This proves the final statement.
\end{proof}

\begin{lemma}\label{ei-star-U(b)-2}Let $p\in\mathbb{N}, u\in\widetilde{A_{(p-1)}}$ and $e_{\alpha_i}(j)\star u\in\widetilde{A_{(p)}}$. For any $v\in W(\sbo_l)$ and $s\in U(\sbo_l)$ so that
\beqs e_{\alpha_i}(j)\cdot (uv)=(e_{\alpha_i}(j)\star u)v+s \eeqs
one has
\beqs s\in\sum_{q=0}^{p-2}\widetilde{A_{(q)}}\otimes W(\sbo_l).\eeqs
\end{lemma}
\begin{proof}
By Lemma \ref{n-dot-u-Z-lemma} and Lemma \ref{dot-tau-commute-lemma} one has
\beqs e_{\alpha_i}(j)\cdot(uv)=(e_{\alpha_i}(j)\cdot u)v=(e_{\alpha_i}(j)\star u)v+wv\eeqs
where $w\in U_{(p-2)}$. Then this lemma follows from Theorem \ref{U(b)-decom-A-W-homegeneous-theorem}.\end{proof}

\section{Whittaker modules}

In this section we classify all nonsingular Whittaker modules over Takiff algebra $\sg_l$. In Subsection 5.1, the proof has a little difference from that of \cite{Kos3}. In Subsection 5.2, the proof for all main results is just the duplication of those in \cite{Kos3}. So we write detailed proof in Subsection 5.1 and we only list main results without proof in Subsection 5.2.

\subsection{}
For any $U(\sg_l)$-module $V$ we let $U(\sg_l)_V$ be the annihilator of $V$. Then $U(\sg_l)_V$ defines an ideal $T_V=T\cap U(\sg_l)_V$ for any subalgebra $T\subseteq U(\sg_l)$.

Now let $V$ be any $U(\sg_l)$-module. For any $u\in U(\sg_l), v\in V$ the action is denoted by $uv\in V$. A vector $w\in V$ is called a Whittaker vector of type $\eta$ if
\beqs xw=\eta(x)w,\;\forall x\in\sn_l.\eeqs
A $U(\sg_l)$-module $V$ is called a Whittaker module if it is cyclically generated by a Whittaker vector.

Now we assume that $V$ is a Whittaker module generated by a Whittaker vector $w$. Let $U(\sg_l)_w$ be the annihilator of $w$. Then $U(\sg_l)_w$ is a left ideal, $U(\sg_l)_V$ is a two-sided ideal and $U(\sg_l)_V\subseteq U(\sg_l)_w$. Clearly one has $V\cong U(\sg_l)/U(\sg_l)_w$ as $U(\sg_l)$-modules.

\begin{lemma}\label{U-w(b)=A-times-W-V}Let $X=\{v\in U(\sbo_l)|(x\cdot v)w=0\;for\;all\;x\in\sn_l\}$, $W(\sbo_l)_V=(Z(\sg_l)_V)^\eta$ and $U(\sbo_l)_w=U(\sbo_l)\cap U(\sg_l)_w$.  Then
\beqs X=(\widetilde{A}\otimes W(\sbo_l)_V)+W(\sbo_l)\eeqs
Moreover, one has
\beqs U(\sbo_l)_w=\widetilde{A}\otimes W(\sbo_l)_V\subseteq X.\eeqs
\end{lemma}
\begin{proof}
For any $x\in\sn$ and $v\in W(\sbo)_l$, one has by \eqref{n-dot-Z=0} that
\beqs (x\cdot v)w=((x\cdot1)v)w=((x^\eta-\eta(x))v)w=0.
\eeqs
Thus $W(\sbo_l)\subseteq X$.

For any $x\in\sn_l, u\in Z(\sg_l)_V$ and $v\in U(\sbo)$ one has
\beqs x\cdot(vu^\eta)=(x\cdot v)u^\eta\eeqs
by Lemma \ref{n-dot-u-Z-lemma}. Clearly $u$ and hence $u^\eta$ is in $U_w$. Thus $vu^\eta\in X$ and therefore
$\widetilde{A}\otimes W(\sbo)_V\subseteq X$ and
hence
\beqs (\widetilde{A}\otimes W(\sbo_l)_V)+W(\sbo_l)\subseteq X.\eeqs

Let $\overline{W}$ be any fixed complement subspace of $W(\sbo_l)_V$ in $W(\sbo)$. Denoted by
\beqs M_r=\widetilde{A}_{(r)}\otimes \overline{W},\eeqs
then one has
\beqs M=\bigoplus_{r=1}^\infty M_r,\eeqs
and
\beqs U(\sbo_l)=M\oplus (\widetilde{A}\otimes W(\sbo_l)_V+W(\sbo_l)).\eeqs

Let $M_{[k]}=\sum_{r=1}^kM_r$ so that the spaces $\{M_{[k]}|k\geq1\}$  are a filtration of $M$. Assume $X\cap M\not=0$ and let $k$ be the minimal such that there exists $0\not=y\in X\cap M_{[k]}$. Then $y$ can be uniquely written as
\beqs y=\sum_{r=1}^ky_r\eeqs
such that $y_r\in\widetilde{A}_{(r)}\otimes W(\sbo)$. Of course one has $y_k\not=0$.

Noting that $\widetilde{A}$ is an $\sn_l$-module with respect to $x\star v$. For each $1\leq i\leq n, 0\leq j\leq l$, let $C_{i,j}=\{v\in\widetilde{A}_{(k)}|e_{\beta_i}(j)*v=0\}$. One has
\beqs \bigcap_{i=1}^{n}\bigcap_{j=0}^lC_{i,j}=0\eeqs
by Lemma \ref{ei-star-U(b)}. Hence there exists $(i,j)$ such that $y_k\not\in C_{i,j}\otimes\overline{W}$.

For any $q\in\mathbb{N}$ let
\beqs \widetilde{A}_{[q]}=\sum_{p\leq q}\widetilde{A}_{(p)}.\eeqs
By Lemma \ref{ei-star-U(b)-2} one has $e_{\beta_i}(j)\cdot y_r\in\widetilde{A}_{[k-2]}\otimes W(\sbo)$ for $1\leq r\leq k-1$.

If $D_{i,j}$ is the linear complement to $C_{i,j}$ in $\widetilde{A}_{(k)}$ we may write $y_k=u_k+v_k$ such that
\beqs u_k\in C_{i,j}\otimes \overline{W}, v_k\in D_{i,j}\otimes\overline{W}\eeqs
and $v_k\not=0$.  But now $e_{\beta_i}(j)\cdot u_k\in\widetilde{A}_{[k-2]}\otimes W(\sbo)$ by Lemma \ref{ei-star-U(b)-2}. On the other hand the map
$D_{i,j}\rightarrow \widetilde{A}_{(k-1)}$ given by $u\mapsto e_{\beta_i}(j)*u$ is injective. Thus one has $e_{\beta_i}(j)\cdot v_k=z_{k-1}+s'$ such that $s'\in\widetilde{A}_{[k-2]}\otimes W(\sbo)$ and
\beqs0\not=z_{k-1}\in\widetilde{A}_{(k-1)}\otimes\overline{W}=M_{k-1}.\eeqs
So we can assume
\beq\label{ei-y-s} e_{\beta_i}(j)\cdot y=z_{k-1}+s\eeq
for some $s\in\widetilde{A}_{[k-2]}\otimes W(\sbo)$.

It is clear that $W(\sbo_l)\cap U_w(\sbo_l)\supseteq W(\sbo_l)_V$. If $v\in W(\sbo_l)\cap U_w(\sbo_l)$ then it can be uniquely written as $v=u^\eta$ for some $u\in Z(\sg_l)$. But $vw=0$ implies $uw=0$ and hence $u\in Z(\sg_l)\cap U_w(\sbo_l)$. Since $V$ is cyclically generated by $w$, this implies $v=u^\eta\in W(\sbo_l)_V$. Thus one has
\beqs W(\sbo_l)\cap U_w(\sbo_l)=W(\sbo_l)_V.\eeqs

By the definition of $X$ one has $e_{\beta_i}(j)\cdot y\in U_w(\sbo_l)$.  Return to \eqref{ei-y-s}.

If $k=1$, then $s=0$ and $0\not=z_0=e_{\beta_i}(j)\cdot y\in\overline{W}\cap U_w(\sbo)$ but which is zero. One obtains a contradiction. Then $k>1$ and
$e_{\beta_i}(j)\cdot y\in X$. Write $W(\sbo)=\overline{W}\oplus W(\sbo)_V$ we may write $s=s_1+s_2$ where
$s_1\in M_{[k-2]}$ and $s_2\in (\widetilde{A}\otimes W(\sbo_l)_V)+W(\sbo_l)$. Put $M_{[0]}=0$. Thus $e_{\beta_i}(j)\cdot y=t_2+s_2$ where $t_2=z_{k-1}+s_1$. Then
$t_2=e_{\beta_i}(j)\cdot  y-s_1\in X$ since $e_{\beta_i}(j)\cdot y\in X$ and $s_2\in \widetilde{A}\otimes W(\sbo_l)_V+W(\sbo_l)\subseteq X$. However, $0\not=t_2\in M_{[k-1]}$. Thus $0\not=t_2\in X\cap M_{[k-1]}$. Again one obtains a contradiction. Hence $X=\widetilde{A}\otimes W(\sbo_l)_V+W(\sbo_l)$.

Applying $\rho^\eta$ to $U(\sg_l)Z(\sg_l)_V\subseteq U_w$ one has $\widetilde{A}\otimes W(\sbo_l)_V\subseteq U_w(\sbo_l)$.  By $W(\sbo_l)\cap U_w(\sbo_l)=(Z(\sg_l)_V)^\eta=W(\sbo_l)_V$, one has
\beqs U_w(\sbo_l)=\widetilde{A}\otimes W(\sbo_l)_V.\eeqs
\end{proof}

\begin{theo}Let $V$ be any Whittaker module over $U(\sg_l)$ cyclically generated by a Whitaker vector $w$ and let $U(\sg_l)_w$ be the annihilator of $w$. Then one has
\beqs U(\sg_l)_w=U(\sg_l)Z(\sg_l)_V+U(\sg_l)U(\sn_l)_\eta.\eeqs
\end{theo}
\begin{proof}It is clear that $U(\sg_l)Z(\sg_l)_V+U(\sg_l)U(\sn_l)_\eta\subseteq U(\sg_l)_w$.

Through the same process in \eqref{lemma-Z-W-ep} one can prove $(uz)^\eta=u^\eta z^\eta$. Thus $(U(\sg_l)Z(\sg_l))^\eta=U(\sbo_l)W(\sbo_l)$. Noting that the map $\rho_\eta$ is the projection with kernel $U(\sg_l)U_\eta(\sn_l)$ and it induces the isomorphism $Z(\sg_l)\rightarrow W(\sbo_l)=Z(\sg_l)^\eta$. The ideal $Z(\sg_l)_V$ of $Z(\sg_l)$ is mapped to an ideal $(Z(\sg_l)_V)^\eta=W(\sbo_l)_V$ of $W(\sbo_l)$. By Theorem \ref{U(b)-decom-A-W-homegeneous-theorem} and \eqref{decomposition-U(g)} one has
\beqs &&U(\sg_l)Z(\sg_l)_V+U(\sg_l)U(\sn_l)_\eta=U(\sg_l)W(\sg_l)_V+U(\sg_l)U(\sn_l)_\eta\\
&&=U(\sbo_l)W(\sg_l)_V+U(\sg_l)U(\sn_l)_\eta=\widetilde{A}\otimes W(\sg_l)_V+U(\sg_l)U(\sn_l)_\eta.\eeqs
But then by Lemma \ref{U-w(b)=A-times-W-V}
\beqs U(\sg_l)_w=\widetilde{A}\otimes W(\sg_l)_V\subseteq U(\sg_l)Z(\sg_l)_V+U(\sg_l)U(\sn_l)_\eta.\eeqs
We have finished the proof.\end{proof}

\subsection{} Now we can classify all Whittaker $\sg_l$-modules  up to equivalence. The following results respectively are analogues of Theorems 3.2-3.6 in \cite{Kos3}.

\begin{theo}Let $V$ be any Whittaker module over $U(\sg_l)$, let $U(\sg_l)_V$ be the annihilator of $V$ and let $Z(U(\sg_l))$ be the center of $U(\sg_l)_V$. Then the correspondence
\beqs V\mapsto Z(U(\sg_l))_V=Z(U(\sg_l))\cap U(\sg_l)\eeqs
sets up a bijection between the set of all equivalence classes of Whittaker modules and the set of all ideals of $Z(U(\sg_l))$.\end{theo}

\begin{theo}Let $V$ be any module over $U(\sg_l)$. Then $V$ is a Whittaker  module if and only if one has an isomorphism
\beqs V\cong U(\sg_l)\otimes_{Z(U(\sg_l))\otimes U(\sn_l)}(Z(U(\sg_l)/Z_*) \eeqs
of $U(\sg_l)$-modules for some ideal $Z_*$ of $Z(U(\sg_l))$. Moreover in such a case the ideal $Z_*$ is uniquely given by $Z_*=Z(U(\sg_l))_V$.
\end{theo}

\begin{theo}Let $V$ be any Whittaker module over $U(\sg_l)$ cyclically generated by a Whitaker vector $w$. Then $v\in V$ is a Whittakker vector if and only if $v=uw$ for some element $u\in Z(U(\sg_l))$.
\end{theo}

\begin{theo}Assume that  $V$ is a  Whittaker module over $U(\sg_l)$. Then  one has an isomorphism
\beqs {\rm End}_{U(\sg_l)}V\cong Z(U(\sg_l))/Z(U(\sg_l))_V\eeqs
of commutative algebras.\end{theo}

\begin{theo}Let $V$ be a Whittaker module over $U(\sg_l)$ of nonsingular type $\eta$. The following conditions are equivalent:
\begin{itemize}
\item[(1)]{$V$ is irreducible.}
\item[(2)]{$V$ admits an infinitesimal character.}
\item[(3)]{The annihilator of $V$ in $Z(U(\sg_l))$ is a maximal ideal of $Z(U(\sg_l))$.}
\item[(4)]{The space of Whittaker vectors in $V$ is one-dimensional.}
\item[(5)]{All nonzero Whittaker vectors in $V$ are cyclically vectors.}
\item[(6)]{The centralizer ${\rm End}_{U(\sg_l)}V$ reduces to the constants $\mc$.}
\item[(7)]{$V$ is isomorphic to $Y_{\eta,\xi}$ for some central character $\xi$.}
\end{itemize}
\end{theo}

\section{The symplectic structure of $Z$ of the normalized Jacobi elements}


\subsection{}
From now on we assume that the ground field is the real number field $\mr$ except special emphasis. In this section one considers  the  $2n(l+1)$-dimensional manifold
\beq\label{Z}Z=\ssf+\sum_{1\leq i\leq n\atop 0\leq j\leq l}\mr h_i(j)+\sum_{1\leq i\leq n\atop 0\leq j\leq l}\mr^+e_{\alpha_i}(j).\eeq
As an analogue of that in \cite{Kos5} for real split semi-simple Lie algebras, it is referred to the manifold of  normalized Jacobi elements of $\sg_l$. For any $y\in Z$ the tangent space to $Z$ at $y$ is clearly given by
\beqs T_y(Z)=\sum_{1\leq i\leq n\atop 0\leq j\leq l}\mr h_i(j)+\sum_{1\leq i\leq n\atop 0\leq j\leq l}\mr e_{\alpha_i}(j).\eeqs

In particular we choose the root vectors $e_{\alpha}$ so that $\kappa(e_{\alpha},e_{\beta})=\delta_{\alpha,-\beta}$ and let $h_i=[e_{\alpha_i},e_{-\alpha_i}]$. Thus $\kappa(h_i,h_i)=2$ for any $i$ and this follows a unique bilinear form  on $\sh^*$ also denoted by $\kappa$ such that $\kappa(\alpha_i,\alpha_i)=2$. Noting that $(c_{i,j})$ is the Cartan matrix, then $\kappa(\alpha_i,\alpha_j)=c_{i,j}$ and $\kappa(h_i,\omega_j)=\delta_{i,j}$ for any pair $(i,j)$.


Let $G={\rm Aut}\sg$ be the adjoint group of $\sg$. Of course $\sg_l$ is real split since $\sg$ is real split simple. Also let $\sg=\sk+\sp$ be a Cartan decomposition of $\sg$ where $\sk$ is the Lie algebra of a maximal compact subgroup of $G$ and $\sp$ is the orthogonal complement to $\sk$ with respect ${\rm Re}\kappa$. Let $\theta$ be the corresponding Cartan involution of $\sg$ which acts as identity on $\sk$ and negative identity on $\sp$.  For any $x\in\sg$ let $x^*=-\theta(x)$. Now $\sg$ becomes to a Hilbert space with respect to the inner product $\kappa_*$ which is defined by $\kappa_*(x,y)=\kappa(x,y^*)$.

Now we define an involution of $\sg_l$ by
\beq\label{involution} x(j)\mapsto (x(j))^*=x^*(l-j)\eeq
for any $x\in\sg, 0\leq j\leq l$. Of course  one defines an inner product $Q_*$ on $\sg_l$ by putting
\beq\label{inner-product}Q_*(x,y)=Q(x,y^*).\eeq
Clearly $\sg_l$ is a Hilbert space with respect to $Q_*$. One notes that the positive definiteness of $Q_*$ follows from that of $\kappa_*$ and  $Q_*(x,y)=\overline{Q_*(y,x)}$ follows from $\kappa_*(x,y)=\overline{\kappa_*(y,x)}$, where $^-$ denote the conjugation.

\begin{remark}It should be pointed out that most results in \cite{Kos5} are proved under assumption that the Lie algebra is real split semi-simple (or complex semi-simple). However the proof process only requires that the Lie algebra has a non-degenerated invariant bilinear form which induces a structure of a Hilbert space. Thus the corresponding results are true for $\sg_l$ here.
\end{remark}

For any subspace $\sa\subseteq\sg$ let $\sa^*=\{x\in\sg_l:x\in\sa\}, \sa^\bot=\{x\in\sg_l:Q_*(\sa,x)=0\}$ and $\sa^0=\{x\in\sg_l:Q(\sa,x)=0\}$.
The three maps $\sa\rightarrow \sa^*,\sa^\bot,\sa^0$ commute with any another and any one is the composite of the other two ones. In particular, one has
\beqs (\sbo_l)^*=\sb_l, (\sbo_l)^\bot=\sn_l, (\sbo_l)^0=\overline{\sn}_l.\eeqs

\subsection{}
%
%
The symmetric algebra $S(\sg_l)$ can be regarded as the algebra of  polynomial functions on $\sg_l$. For any $v\in S(\sg_l)$ let $\delta v$ denote the differential of $v$ when $v$ is regarded as a function on $\sg_l$.

\begin{prop}[see Proposition 1.2.1 in \cite{Kos5}] The Possion structure of $S(\sg_l)$ is given by
\beqs [u,v](z)=Q(z,[(\delta u)(z),(\delta v)(z)])\eeqs
for any $u,v\in S(\sg_l)$ and $z\in\sg_l$.
\end{prop}


%
\begin{theo}[see Theorem 1.5 in \cite{Kos5}]\label{Possion-commute}
The elements $I^\ssf, J^\ssf$ in $S(\sbo_l)$ Poisson commute for any pair of invariants $I,J\in S(\sg_l)$.\end{theo}

Let $P_{\sbo_l}:\sg_l\rightarrow \sbo_l$ be the orthogonal projection of $\sg_l$ on $\sbo_l$ with respect to the inner product $Q_*$ and let $\delta_{\sbo_l}$ be map $\sg_l\rightarrow\sbo_l$ defined by
\beqs (\delta_{\sbo_l}u)(x)=P_{\sbo_l}((\delta u)(x)).\eeqs
\begin{prop}[see Proposition 1.2 in \cite{Kos5}]$\delta_{\sbo_l}$ defines a Poisson structure of $S(\sbo_l)$ which coincides with the usual Possion structure of $S(\sbo_l)$ defined by the Lie algebra structure of $\sbo_l$.
\end{prop}

One considers the map
\beq\label{variety-translation}\tau_\ssf:\sb_l\rightarrow\ssf+\sb_l,\;x\mapsto \ssf+x,\eeq
which is an isomorphism of affine varieties.

\begin{lemma}[see Lemma 1.6.1 in \cite{Kos5}]\label{I-y-delta}For any $I\in S(\sg_l)^{G_l}$ and $y\in\ssf+\sb_l$ one has
\beqs [y,(\delta_{\sbo_l}I)(y)]\in\sb_l.\eeqs
\end{lemma}

For any $y\in\ssf+\sb_l$, the tangent space to $\ssf+\sb_l$ at $y$ is clearly $\sb_l$. Then for any $I\in S(\sg_l)^{G_l}$, the map $y\mapsto [y,(\delta_{\sbo_l}I)(y)]$ is a polynomial map on $\sg_l$. By Lemma \ref{I-y-delta} there exists a smooth vector field $\chi_I$ on $\ssf+\sb_l$ such that
\beq\label{vector-field-sb}(\chi_I)_y=[y,(\delta_{\sbo_l}I)(y)],\;\forall y\in\ssf+\sb_l.\eeq

Applying the vector field $\tau_\ssf^{-1}\chi_I$ to $u\in S(\sbo_l)$ one has
\beqs ((\tau_\ssf^{-1}\chi_I)u)(x)=[I^\ssf,u](x)\eeqs
for any $x\in\sb_l$.
\begin{theo}[see Theorem 1.6.1, Theorem 1.6.2 in \cite{Kos5}]\label{commutativity-invariants}For any $I,J\in S(\sg_l)^{G_l}$ one has $\chi_IJ=0$ on the variety $\ssf+\sb_l$ and $[\chi_I,\chi_J]=0$. \end{theo}

\subsection{} In this subsection we temporarily assume that $\sg_l$ is a complex Lie algebra. Let $G(0)\subseteq G_l$ be the Lie subgroup corresponding to $\sg(0)$ and $G[1]\subseteq G_l$ be the Lie subgroup corresponding to $\sg[1]$. Clearly $G[1]$ is nilpotent.
Also let $N(0), H(0)$ be the Lie subgroups corresponding to $\sn(0)$ and $\sh(0)$, respectively. Let $W$ be the Weyl group of $\sg(0)$ with the unique longest element $\omega_L$. Then the Bruhat-Gelfand decomposition of $G(0)$ asserts that
\beqs G(0)=\bigcup_{\sigma\in W}N(0)s(\sigma)H(0)N(0),\eeqs
where $s(\sigma)\in G(0)$ is the unique element in the normalizer of $\sh(0)$, corresponding to $\sigma$, such that
\beqs s(\sigma)e_{\alpha_i}(0)=e_{\sigma(\alpha_i}(0).\eeqs
Noting that $s(\omega_L)N(0)s(\omega_L)^{-1}=\overline{N}(0)$ is the subgroup corresponding to $\overline{\sn}(0)$. One knows that
\beqs G(0)_*=\overline{N}(0)H(0)N(0)\eeqs
is a Zariski open subset of $G(0)$ and the map
\beqs \overline{N}(0)\times H(0)\times N(0)\rightarrow G(0)_*, \;(x,h,y)\mapsto xhy\eeqs
is an isomorphism of algebraic varieties. The set $\overline{N}(0)H(0)N(0)$ is also called the big Bruhat cell of $G(0)$.

Now clearly  $G_l$ has a decomposition
\beqs G(0)=\bigcup_{\sigma\in W}N(0)s(\sigma)H(0)N(0)G[1],\eeqs
and for any $j$ one has
\beqs s(\sigma)e_{\alpha_i}(j)=e_{\sigma(\alpha_i)}(j).\eeqs
Let $(G_l)_*=G(0)_*G[1]$ then one has
\beqs (G_l)_*=\overline{N}_lH_lN_l,\eeqs
which is a Zariski open subset of $G_l$. Where $\overline{N}_l, H_l, N_l$ are the Lie subgroups corresponding to $\overline{\sn}_l,\sh_l,\sn_l$, respectively. In particular, the map
\beqs \overline{N}_l\times H_l\times N_l\rightarrow (G_l)_*, \;(x,h,y)\mapsto xhy\eeqs
is also an isomorphism of algebraic varieties.

For any
\beqs y=\ssf+\sum_{i,j}a_{i,j}h_{i}(j)+\sum_{i,j}b_{i,j}e_{\alpha_i}(j)\in Z,\eeqs
assume $y=\sum_{j=0}^ly_j(j)$ such that $y_j(j)\in\sg(j)$. Clearly $y_0(0)$ is regular in $\sg(0)$ then there exits $g\in G(0)$ such that
$gy\in h+\sg[1]$ for some regular element $h\in\sh(0)$ of $\sg(0)$. Then it is easy to see that there exists $g'\in G[1]$ such that $g'gy\in h+\sh[1]$.

For any $x\in\sg_l$ let $(G_l)^x$ denote the centralizer of $x$ in $G_l$. Then one has
\beqs (G_l)^y\cong (G_l)^{g'gy}=H_l,\eeqs
which is an abelian connected algebraic group of complex dimension $n(l+1)$.

\begin{prop}\label{abelian-algebraic-group}As algebraic groups one has the isomorphism
\beqs (G_l)^y\cong (\mc^*)^p\times(\mc^*)^q\eeqs
for any $y\in Z$ and some $p,q$ (dependent on $y$) where $p+q=n(l+1)$.\end{prop}
\begin{proof}This follows from the well known classification of connected abelian complex algebraic groups.
\end{proof}

For any $y\in Z$ let
\beqs (G_l)^y_*=(G_l)^y\cup(G_l)_*\eeqs
so that $(G_l)^y_*$ is a Zariski open subset of $(G_l)^y$, which is  open and  dense in $(G_l)^y$.

\begin{remark}
By the proof for $(G_l)^y$ it is clear that $(G_l)^y\cong H_l$ if $y_0(0)$ is regular in $\sg(0)$. We do not know whether the regularity of $y_0(0)$ is a necessary condition.  However, the proof process in \cite{Kos5} in fact only requires the property of Proposition \ref{abelian-algebraic-group}. Thus we can use the corresponding results in \cite{Kos5}. Clearly $G_l$ also has the decomposition
\beqs G_l=\bigcup_{\sigma\in W}\overline{N}_ls(\sigma)H_lN_l,\eeqs
called the  Bruhat-Gelfand decomposition of $G_l$. This also provides a necessary condition which is used in the proof in \cite{Kos5}.  One refers to $\overline{N}_l\times H_l\times N_l$ the big Bruhat cell of $G_l$ as the analogue of that of $G(0)$.
\end{remark}

\subsection{}

It is known  that $(\sbo_l)^0=\sno_l$ and the sum
\beq\label{direct-sum-g=f+b}\sg=\sno_l\oplus(\ssf+\sb_l)\eeq
is direct. Let
\beq\label{projection-g-f+b}P:\sg_l\rightarrow\ssf+\sb_l,\;x\mapsto \ssf+x \eeq
be the projection on $\ssf+\sb_l$ according to the decomposition \eqref{direct-sum-g=f+b}.

\begin{prop}\label{symplectic-f+b}Regarding $\sb_l$ as a $\overline{B}_l$-manifold with respect to the coadjoint representation. Then the map
\beqs\sb_l\mapsto \ssf+\sb_l,\;x\mapsto \ssf+x\eeqs
is an isomorphism of  $\overline{B}_l$-manifolds.
\end{prop}
\begin{proof}It follows from the fact
\beqs Q(\ssf+x,y)=Q(x,y)\eeqs
for any $x\in \sb_l$ and $y\in\sbo_l$.\end{proof}

Proposition \ref{symplectic-f+b} implies that any orbit of $\overline{B}_l$ in $\ssf+\sb_l$ under the action
\beqs a\cdot y=Pay\eeqs
has a sympletic structure. By the definition of $Z$ as a submanifold of $\ssf+\sb_l$ one has a global coordinate system $\rho_{i,j},\gamma_{i,j}\in C^\infty(Z)$ and any $y\in Z$ is of the form
\beqs y=\ssf+\sum_{1\leq i\leq n\atop 0\leq j\leq l}\rho_{i,j}(y)h_i(j)+\sum_{1\leq i\leq n\atop 0\leq j\leq l}\gamma_{i,j}(y)e_{\alpha_i}(j),\eeqs
where $\rho_{i,j}(y)\in\mr, \gamma_{i,j}(y)\in\mr^+$.

\begin{prop}
$Z$ is an orbit of $\overline{B}_l$ in $\ssf+\sb_l$, then $Z$ has a symplectic structure given by
\beq\label{omega-Z} \omega_Z=\sum_{i=1}^n\sum_{j+s\geq l}(t_{i;j,s})d\rho_{i,j}\wedge d\gamma_{i,s},\eeq
where $T_i=(t_{i;j,s})$ is the inverse matrix of the $(l+1)\times(l+1)$-matrix
\beqs \left[t'_{i;j,s}=\gamma_{i,j+s-l}\right].\eeqs
Moreover if $\chi_{I|_Z}$ is the Hamiltonian vector field corresponding to $I_Z\in C^\infty(Z)$ for any $I\in S(\sg_l)^{G_l}$ then one has
\beqs \chi_{I|_Z}=\chi_I.\eeqs
\end{prop}
\begin{proof}Let $y=\ssf+{\bf e}\in Z$ where ${\bf e}=\sum_{1\leq i\leq n, 0\leq j\leq l}e_{\alpha_i}(j)$. If p
\beqs z=\ssf+\sum_{i,j}c_{i,j}h_i(j)+\sum_{i,j}a_{i,j}e_{\alpha_i}(j)\in Z,\eeqs
let $x\in\sno_l$ such that $x+c_{i,j}e_{-\alpha_i}(j)$. Furthermore, we can find real constants $d_{i,j}, 1\leq ,i\leq n, 1\leq j\leq l$ and $a_i>0, 1\leq i\leq n$ such that
\beqs a_{i,j}=a_i\sum_{s\leq 1, ks\leq j} \frac{d_{i,s}^k}{k!}\eeqs
for all $(i,j)$. Let $h'=\sum d_{i,j}\omega_i(j)$ and let $h\in H(0)$ such that $h^{\alpha_i}=a_i$.

Now let $g=h\exp(h')\exp(x)\in\overline{B}_l$. Then
\beqs gy-z\in\sno_l,\eeqs
so $g\cdot y=z$ since $\sno_l=(\sbo_l)^0$.

For any $u\in S(\sbo_l)$ let $\psi^u\in C^\infty(Z)$ be defined by $\psi^u=u|_Z$. Then one has
\beq\label{bracket-omega}[\psi^u,\psi^v]=\psi^{[u,v]} \eeq
for any $u,v\in S(\sbo_l)$.

Noting that $Q(\omega_i(j),h_r(s))=\delta_{i,r}\delta_{j,l-s}$ and $Q(x(j),h_r(s))=\delta_{j,l-s}\alpha_i(x)$ for any $x\in\sh$. Thus one has
\beqs [e_{\alpha_i}(j),\omega_r(s)]=\delta_{i,r}\delta_{j,l-s}e_{\alpha_i}(j+s).\eeqs
Furthermore one notes
\beqs \psi^{\omega_i(j)}=\rho_{i,l-j},\;\psi^{e_{-\alpha_i}(j)}=\gamma_{i,l-j}.\eeqs
Thus one has in $C^\infty(Z)$ that
\beqs [\gamma_{i,j},\rho_{r,s}]=\delta_{i,r}\gamma_{i,j+s-l}.\eeqs
Moreover one has $[\gamma_{i,j},\gamma_{r,s}]=[\rho_{i,j},\rho_{r,s}]=0$ since $[e_{-\alpha_i}(j),e_{-\alpha_r}(l)]\in[\sn_l,\sn_l]$ and $[\omega_i(j),\omega_r(s)]=0$. By Proposition 6.1 in \cite{Kos5} and \eqref{bracket-omega} we compute $\omega_Z$ in terms the coordinate system $\rho_{i,j},\gamma_{i,j}$, then it follows \eqref{omega-Z}.

The proof for the final statement is same to that of Proposition 6.4 in \cite{Kos5}.
\end{proof}

\section{Hyperbolic Toda lattices}

\subsection{}

The Hamiltonian system which is a dynamical system and completely described by a function $H(x,t)=H(p,q,t)$. This function is called a Hamiltonian.  In the mechanics background $x=(p,q)$ stands for the state  of the dynamical system, where $p, q$ are vectors of the same dimension, denoting the momentum and position, respectively.

Applying to a system of particles moving on a line with respect to some potential, we envision $n$ particles with masses $m_i, i=1,\cdots,n$.  Let $\mr^{2n}$ be the phase space with linear canonical coordinate functions  $p_i,q_j, i,j=1,\cdots,n$ so that any $x\in\mr^{2n}$ is a state of the system and $p_i(x), q_j(x)$ are the momentum and the position of the $i$th particle, respectively. A potential function $U$ is a smooth function in $C^\infty(\mr^{2n})$ which depends only on the $q$'s. The Hamiltonian $H\in C^\infty(\mr^{2n})$ of the system is given by
\beq\label{hamiltonian}H=\sum_{i=1}^n\frac{p_i^2}{2m_i}+U,\eeq
which describes the evolution in time of a mechanical system with total energy $H$. For any $x\in\mr^{2n}$ we assume that $x(t)\in\mr^{2n}$ is the state of the system at time $t$ such that $x(0)=x$. Then $x(t)$ is determined by the Hamilton's equation
\beq\label{hamiltonian-equation}\frac{dp_j(x(t))}{dt}=-\frac{\partial H}{\partial q_j}(x(t)),&&\frac{dq_j(x(t))}{dt}=\frac{\partial H}{\partial p_j}(x(t)).\eeq
In particular, when $m_i=1$ and
$$U=\sum_{i=1}^{n-1}e^{q_i-q_{i+1}}$$
the system is called the Toda lattice.

As usual one deals with the equation \eqref{hamiltonian-equation} by introducing a symplectic structure determined by a closed non-singular $2$-form
\beq\label{2-form}\omega=\sum_{i=1}^ndp_i\wedge dq_i.\eeq
This $2$-form defines a Poisson bracket $[,]$ which makes $C^\infty(\mr^{2n})$ to be a Lie algebra. In particular one has $[p_i,q_j]=\delta_{i,j}$. Regarding $p_i$ as the partial derivative $\frac{\partial}{\partial q_i}$ then $H$ becomes an elliptic differential operator on functions in the variables $q_i$ since all $m_i$ are positive.

Furthermore there exists a symmetric bilinear form $B_{H}$ such that
\beq\label{bilinear-form}B_{H}(\varphi,\psi)1=[\varphi,[\psi,H]],\;\forall \varphi,\psi\in C^\infty(\mr^{2n}).\eeq
In particular one has
\beq\label{bilinear-form-q-q}B_{H}(q_i,q_j)=\frac{\delta_{i,j}}{m_i},\eeq
which defines a positive definite symmetric bilinear form on the space $\mathcal{Q}=\sum_{i=1}^n\mr q_i$.

Now considering $\sg$ as a simple Lie algebra of rank $l$ with simple roots $\alpha_1,\cdots,\alpha_l$ (corresponding to a Cartan subalgebra $\sh$). Assume that
\beqs U=\sum_{i=1}^lr_ie^{\psi_i}\eeqs
for some positive constants $r_i$ and functions $\psi_i\in\mathcal{Q}$. We say that the system is associative to the Lie algebra $\sg$ if
\beq\label{bilinear-form-dynkin} B_{H}(\psi_i,\psi_j)=\kappa(\alpha_i,\alpha_j)\eeq
for all $i,j=1,\cdots,l$, where $\kappa$ is a positive multiple of the Killing form. In this case the Hamiltonian system sometimes is called the Bogoyavlensky-Toda lattice which was introduced by Bogoyavlensky in \cite{B}.


\subsection{} Return to the system of particles.  Now we only focus on the fundamental invariant $I=\frac12Q$ of $S(\sg_l)^{G_l}$. Explicitly one has
\beq\label{I-explicit}I=\frac12\sum_{i=1}^n\sum_{j=0}^lh_i(j)\omega_i(l-j)+\frac12\sum_{\alpha\in\Delta}^n\sum_{j=0}^le_{-\alpha}(j)e_{\alpha}(l-j).\eeq
Thus for any
\beqs y=\ssf+\sum_{1\leq i\leq n\atop 0\leq j\leq l}\rho_{i,j}h_i(j)+\sum_{1\leq i\leq n\atop 0\leq j\leq l}\gamma_{i,j}e_{\alpha_i}(j),\eeqs
one has
\beqs I(y)=\frac12Q(y,y)=\frac12\sum_{i,r=1}^n\sum_{j=0}^lc_{i,r}\rho_{i,j}\rho_{r,l-j}+\sum_{i=1}^n\sum_{j=0}^l\gamma_{i,j},\eeqs
where $(c_{i,r})$ is the Cartan matrix..

Next one considers the case  $l=1$. Thus
\beqs I(y)=\frac12\sum_{i,r=1}^n\sum_{j=0}^lc_{i,r}\rho_{i,0}\rho_{r,1}+\sum_{i=1}^n(\gamma_{i,0}+\gamma_{i,1}).\eeqs
The matrix $[t'_{i;j,s}]$ is
\beqs\left[\begin{array}{cc}0&-\gamma_{i,0}\\-\gamma_{i,0}&-\gamma_{i,1}\end{array}\right]\eeqs
and hence
\beqs T_i=\left[\begin{array}{cc}{-\gamma_{i,1}}/{\gamma_{i,0}}^2&{\gamma_{i,0}}^{-1}\\{\gamma_{i,0}}^{-1}&0\end{array}\right].\eeqs
The symplectic structure of manifold $Z$ then is given by
\beqs \omega_Z=\sum_{i=1}^n(-\frac{\gamma_{i,1}}{{\gamma_{i,0}}^2}d\rho_{i,0}\wedge d\gamma_{i,0}+\frac{1}{{\gamma_{i,0}}}d\rho_{i,0}\wedge d\gamma_{i,1}+\frac{1}{{\gamma_{i,0}}}d\rho_{i,1}\wedge d\gamma_{i,0}).\eeqs
Let $\gamma_{i,0}=\exp(\phi_{i,0}), \gamma_{i,1}/\gamma_{i,0}=\phi_{i,1}$ for some smooth functions $\phi_{i,0}, \phi_{i,1}$ where $\phi_{i,1}>0$. Then one has
\beq\label{omega-Z-2} \omega_Z=\sum_{i=1}^n(d\rho_{i,0}\wedge d\phi_{i,1}+d\rho_{i,1}\wedge d\phi_{i,0}).\eeq

The corresponding Hamiltonian is
\beq\label{takiff-hamiltonian}H=\sum_{i,j=1}^nc_{i,r}\rho_{i,0}\rho_{j,1}+\sum_{i=1}^n(1+\phi_{i,1})e^{\phi_{i,0}}.\eeq
\begin{theo}The system described by \eqref{takiff-hamiltonian} is completely integrable in the Liouville sense.
\end{theo}
\begin{proof}This follows from Theorem \ref{Possion-commute}.
\end{proof}

{\it Example 1:} (1) Let $n=1$ and denoting
\beqs \left(\frac{\partial}{\partial x}, \frac{\partial}{\partial y}, x,y\right)=\left(\frac{\partial}{\partial \phi_{1,0}}, \frac{\partial}{\partial \phi_{1,1}},\phi_{1,0}, \phi_{1,1}\right),\eeqs
 then one has
\beqs H=2\frac{\partial}{\partial x}\frac{\partial}{\partial y}+(1+y)e^x,\;(x,y)\in\mr\times\mr^+,\eeqs
which is a hyperbolic differential operator over the upper half plane.

(2) Let $n=1$ and denoting
\beqs \left(p, \overline{p}, q,\overline{q}\right)=\left(\frac{\partial}{\partial \phi_{1,0}}, \frac{\partial}{\partial \phi_{1,1}},\phi_{1,0}, \phi_{1,1}\right),\eeqs
 then one obtains
 \beqs H=2p\overline{p}+(1+\overline{q})e^q,\eeqs
which is an explicit example of \eqref{h-example-3} with $U=(1+\overline{q})e^q$. Because $\overline{q}$ denotes the difference of the positions of two particles, here the condition $\overline{q}>0$ can be understood as that the paired particles can't exchange their positions on the line.

The  differential operator corresponding to \eqref{takiff-hamiltonian} is of hyperbolic type which acts on  functions with domain $\mr^n\times(\mr^+)^n$.

In general, one considers the following system:
\beq\label{takiff-toda-1} &&H=\sum_{i=1}^nc_{i,j}p_i\overline{p_j}+U,\\
\label{takiff-toda-2}&&\frac{\partial H}{\partial p_i}=\frac{d{\overline{q}_i}}{dt},\;\frac{\partial H}{\partial \overline{p}_i}=\frac{d{q}_i}{dt},\;\frac{\partial H}{\partial q_i}=-\frac{d{p}_i}{dt},\;\frac{\partial H}{\partial \overline{q}_i}=-\frac{d\overline{p}_i}{dt},\eeq
where $\overline{q}_i>0$, $(c_{i,j})$ is the Cartan matrix of $\sg$,  the smooth functions ${p}_i, \overline{p}_i, {q}_i, \overline{q}_i\in C^\infty(\mr^{4n})$ are  a linear canonical coordinate system of $\mr^{2n}\times (\mr^n\times(\mr^+)^n)$  and the potential function $U$ is only dependent on the $q_i, \overline{q}_i$.

We refer this system to a hyperbolic Toda lattice corresponding to the Takiff algebra $\sg_1$. Of course, the invariant $I=\frac12Q$ of generalized Takiff algebra $\sg_l$ suggests a more generalized Hamiltonian system.

\subsection{} Now considering the hyperbolic Toda lattice in the above example, then one has
\beqs &&\frac{d\overline{q}}{dt}=2\overline{p},\\
 &&\frac{d^2\overline{q}}{dt^2}=\frac{2d\overline{p}}{dt}=2e^q,\\
 &&\frac{d^3\overline{q}}{dt^3}=\frac{2de^q}{dt}=4pe^q,\\
 &&\frac{d^4\overline{q}}{dt^4}=\frac{4pe^q}{dt}=4e^q\left(2p^2+\frac{dp}{dt}\right)=4e^q\left(2p^2+(1+\overline{q})e^q\right),\eeqs
and hence
\beqs \frac{d^4\overline{q}}{dt^4}\frac{d^2\overline{q}}{dt^2}=\left(\frac{d^3\overline{q}}{dt^3}\right)^2+(1+\overline{q})\left(\frac{d^2\overline{q}}{dt^2}\right)^3.\eeqs
Furthermore, it is obvious that
\beqs \frac{d}{dt}\left({\frac{d^3\overline{q}}{dt^3}}\Big/{\frac{d^2\overline{q}}{dt^2}}\right)=\frac{d}{dt}\left(\overline{q}+\frac{\overline{q}^2}2\right).\eeqs

Thus in order to solving the Hamiltonian equation, by letting  $\overline{q}=2\Psi(t)$ it suffices to solving the ordinary differential equation
\beq\label{ode} \Psi^{'''}(t)=C_0\Psi^{''}(t)+\Psi(t)^2+\Psi(t)\eeq
for any constant $C_0\in\mr$. It is known that  equation \eqref{ode} is always (at least locally) solvable for any initial values
\beqs \Psi(t_0)=C_1, \Psi^{'}(t_0)=C_2, \Psi^{''}(t_0)=C_3.\eeqs

{\it Example 2:} Now we assume that $|C_0|<1$, $t_0=0$ and
\beqs\Psi(t)=\sum_{n=0}^\infty a_nt^n\eeqs
solves \eqref{ode}. 
Then one has
\beqs a_{n+3}=\frac{1}{(n+1)(n+2)(n+3)}\left((n+1)(n+2)C_0a_{n+2}+a_n+\sum_{i=0}^{n}a_ia_{n-i}\right),\;\forall n\geq0.\eeqs
If $a_0, a_1, a_2\in(-1,1)$, then
\beqs &&|a_3|=\frac{1}{6}\left|2C_0a_2+a_0+2a_1a_2\right|<1<\frac2{1^2},\\
&&|a_4|=\frac{1}{24}\left|6C_0a_3+a_1+2a_1a_3+a_2^2\right|<\frac{10}{24}<\frac2{2^2}\\
&&|a_5|=\frac{1}{60}\left|12C_0a_4+a_2+2a_1a_4+2a_2a_3\right|<\frac{1}{6}<\frac2{3^2}.\eeqs
By induction we assume  $|a_{k+2}|<\frac{2}{k^2}$ for $1\leq k\leq n$ where $n\geq3$. Then one has
\beqs |a_{n+3}|&=&\frac{\left|(n+1)(n+2)C_0a_{n+2}+a_n+\sum_{i=1}^{n}a_ia_{n-i}\right|}{(n+1)(n+2)(n+3)}\\
&<&\frac{n+6}{(n+1)(n+2)(n+3)}<\frac2{(n+1)^2}.\eeqs
Hence for any $C_0\in(-1,1)$, the series function $\Psi(t)$ in variable $t$ absolutely  converges to a smooth function on $\mr$ such that
\beqs C_1=\Psi(0)=a_0, C_2=\Psi'(0)=a_1, C_3=\Psi''(0)=2a_2.\eeqs
Thus the initial value problem of \eqref{ode} is global solvable for bounded initial values.

Moreover, if $C_0,C_1,C_2,C_3\in\mr^+$ it follows $\Psi(t)>0$ (or equivalently $\overline{q}>0$) at any time $t>0$, this agrees with the condition of Example 1 (2). Immediately one has the following
\begin{theo}Considering the initial value problem of the dynamical system
\beq\label{dynamical-system} \frac{d{\overline{q}}}{dt}=2\overline{p},\;\frac{d{q}}{dt}=2p,\;\frac{d{p}}{dt}=-(1+\overline{q})e^q,\;\frac{d\overline{p}}{dt}=-e^q, \quad \overline{q}>0,\eeq
with initial values $c_0=\overline{q}(0), c_1=\overline{q}'(0), c_2=\overline{q}''(0), c_3=\overline{q}'''(0)$. If the initial values satisfy
\beqs c_0,c_1,\frac{c_2}2, \frac{2(c_3-c_1^2-c_1)}{c_2}\in(0,1),\eeqs
then \eqref{dynamical-system} has a unique global smooth solution on $[0,\infty)$.
\end{theo}

\section*{Acknowledgments}
The author gratefully acknowledge partial financial supports from the National Natural Science Foundation of China (12171155).

{\bf Data Availability Statement}: Not applicable.

{\bf Conflicts of Interest}: The author declared that he has no conflict of interest to this work.

\def\refname{
\end{document}